\documentclass[a4paper]{amsart}
\usepackage[utf8]{inputenc}
\usepackage[english]{babel}
\usepackage[left=2.5cm,right=2.5cm,top=2.5cm,bottom=2.5cm]{geometry}
\usepackage{amsmath,amssymb,latexsym, amsthm,  enumitem, mathabx}
\usepackage{graphicx, xcolor}
\usepackage{hyperref}

\newcommand{\id}{\,\mathrm{d}}
\newcommand{\iid}{\,\mathrm{Id}}

\DeclareMathOperator{\ad}{ad}
\DeclareMathOperator{\End}{End}

\DeclareMathOperator{\im}{im}
\DeclareMathOperator{\cosech}{cosech}
\DeclareMathOperator{\tr}{tr}

\renewcommand{\exp}{\mathop{\rm exp}\nolimits}

\newtheorem{theorem}{Theorem}[section]
\newtheorem{maintheorem}[theorem]{Main Theorem}
\newtheorem{lemma}[theorem]{Lemma}
\newtheorem{corollary}[theorem]{Corollary}
\newtheorem{proposition}[theorem]{Proposition}
\theoremstyle{definition}
\newtheorem*{rem}{Remark}
\newtheorem{definition}[theorem]{Definition}

\title[Isoparametric Functions on Damek--Ricci Spaces]{Classification of Certain Rational Isoparametric Functions on Damek--Ricci Spaces}
\begin{document}
\author{Bal\'azs Csik\'os}
\address{B.~Csik\'os, Dept.\ of Algorithms and their Applications, Faculty of Informatics, E\"otv\"os Lor\'and University,  P\'azm\'any P\'eter stny.\@ 1/C, H-1117 Budapest, Hungary.}
\email{\href{mailto:balazs.csikos@ttk.elte.hu}{csikosbalazs@inf.elte.hu}}

\author{M\'arton Horv\'ath}
\address{M. Horv\'ath, Dept.\ of Algebra and Geometry, Institute of Mathematics, Budapest University of Technology and Economics, M\H{u}\-egye\-tem rkp.\@ 3., H-1111 Budapest, Hungary.} 
\email{\href{mailto:horvathm@math.bme.hu}{horvathm@math.bme.hu}}
\date{}
\maketitle
\begin{abstract}
We classify isoparametric functions on Damek--Ricci spaces which can be written in terms of the standard coordinates $(v,z,t)$ on the half-space model as a polynomial function divided by $t$. Regular level sets of the functions in our classification encompass almost all previously known examples of isoparametric hypersufaces in Damek--Ricci space and also yield new ones. For the new examples, the focal varieties are determined and the mean curvatures of the regular level sets are expressed as a function of the distance from the focal variety. We also study the exceptional case of tubes about $\mathbb R\mathbf H^k$ in  $\mathbb C\mathbf H^k$, which are isoparametric, but cannot be obtained as the level sets of any function in our classification. We show that these tubes are level sets of a polynomial function divided by $t^2$, and that analogous functions  on Damek--Ricci spaces can be  isoparametric only in the case of $\mathbb C\mathbf H^k$. 
\end{abstract}
\section{Introduction}

A hypersurface in a Riemannian manifold is called isoparametric if its nearby parallel hypersurfaces are of constant mean curvature. Isoparametric hypersurfaces in the real hyperbolic space were classified by \'E.~Cartan \cite{Cartan}. Their classification in the Euclidean space were given by B.~Segre \cite{Segre}. Though all isoparametric hypersurfaces in the Euclidean and hyperbolic spaces are  homogeneous, H.~Ozeki and M.~Takeuchi \cite{Ozeki_Takeuchi_1}, \cite{Ozeki_Takeuchi_2}, and D.~Ferus, H.~Karcher, and H.-F.~M\"unzner \cite{Ferus_Karcher_Munzner} found infinitely many non-homogeneous isoparametric hypersurfaces in spheres. Due to the existence of non-homogeneous examples, classification of isoparametric hypersurfaces in the sphere turned out to be a much more difficult problem, which has been achieved in a sequence of papers by  J.~Dorfmeister and E.~Neher \cite{Dorfmeister_Neher}, T.~E.~Cecil, Q.-S.~Chi, and G.~R.~Jensen \cite{Cecil_Chi_Jensen}, Q.-S.~Chi \cite{Chi1}, \cite{Chi2}, \cite{Chi3}, \cite{Chi4}, and R.~Miyaoka \cite{Miyaoka}, \cite{Miyaoka_errata}.
For a survey of the history of the classification of isoparametric hypersurfaces in spaces of constant curvature, see Q.-S.~Chi \cite{Chi_survey}.

There are many results also on isoparametric hypersurfaces in rank one symmetric spaces. Regular orbits of isometric cohomogeneity one actions are isoparametric, prompting classification of such actions. J.~Berndt and H.~Tamaru \cite{Berndt_Tamaru} classified these actions on complex hyperbolic spaces and on the Cayley hyperbolic plane, while J.~C.~D\'iaz-Ramos, M.~Dom\'inguez-V\'azquez, and A.~Rodr\'iguez-V\'azquez \cite{Diaz-Ramos_Dominguez-Vazquez_Rodriguez-Vazquez} extended this to quaternionic hyperbolic space. Subsequently, J.~C.~D\'iaz-Ramos and  M.~Dom\'inguez-V\'azquez \cite{Diaz-Ramos_Dominguez-Vazquez} constructed non-homogeneous isoparametric hypersurfaces in complex hyperbolic spaces, enabling  J.~C.~D\'iaz-Ramos,  M.~Dom\'inguez-V\'azquez, and V.~Sanmart\'in-L\'opez \cite{Diaz-Ramos_Dominguez-Vazquez_Sanmartin-Lopez} to complete the classification of isoparametric hypersurfaces in complex hyperbolic spaces.

Harmonic manifolds share many properties with flat and rank one symmetric spaces, making them relevant for the study of isoparametric hypersurfaces. Locally harmonic manifolds were introduced by E.~T.~Copson and H.~S.~Ruse \cite{Copson_Ruse} as Riemannian manifolds admitting a non-constant radial harmonic function in a punctured neighborhood of any point. A.~J.~Ledger \cite{Ledger} proved that locally symmetric spaces are locally harmonic precisely when they are flat or of rank one.

A.~Lichnerowicz \cite{Lichnerowicz} conjectured in 1944 that locally harmonic manifolds in dimension 4 must be locally symmetric, asking also if this extends to higher dimensions. A.~G.~Walker \cite{Walker} confirmed this conjecture for dimension 4, and Y.~Nikolayevsky \cite{Nikolayevsky} for dimension 5. Z.~I.~Szab\'o \cite{Szabo} proved it for manifolds with compact universal covering space, while G.~Knieper \cite{Knieper} settled it for compact harmonic manifolds without focal points or with Gromov hyperbolic fundamental groups.

For non-compact cases, Lichnerowicz's conjecture fails in infinitely many dimensions starting at $7$. E.~Damek and F.~Ricci \cite{Damek_Ricci} discovered that certain solvable extensions of Heisenberg-type Lie groups with appropriate metrics create globally harmonic manifolds, that are symmetric only when the Heisenberg-type group has center dimension $1$, $3$, or $7$. Currently, harmonic symmetric spaces and Damek--Ricci spaces are the only known harmonic manifolds. J.~Heber \cite{Heber} showed that simply connected homogeneous harmonic manifolds must be flat, rank one symmetric, or Damek--Ricci spaces. It remains open if there exist non-homogeneous harmonic manifolds.

Up to now, two methods has been found to construct isoparametric hypersurfaces in general Damek--Ricci spaces.  

First J.~C.~D\'iaz-Ramos and  M.~Dom\'inguez-V\'azquez \cite{Ramos_Vazquez} observed that the tubes about certain Lie subgroups of the Damek--Ricci group are parallel isoparametric hypersurfaces.

Later the authors \cite{Csikos_Horvath_isoparametric1} gave an analytical description of another family of isoparametric hypersurfaces, which were called sphere-like hypersurfaces. 
E.~T.~Copson and H.~S.~Ruse \cite{Copson_Ruse} proved that a 
non-compact, complete, simply connected Riemannian manifold is harmonic if and only if its geodesic spheres are isoparametric hypersurfaces. This implies that the horospheres, that can be obtained as the limit of an inflating family of geodesic spheres, are also isoparametric. Sphere-like hypersurfaces are members of the analytic prolongation of such a family of inflating spheres beyond the horospheres. 

To define sphere-like isoparametric hypersurfaces explicitely, the authors used the so-called half-space model of Damek--Ricci spaces, which will be defined in 
Subsection \ref{sec:half_space_model}. Isoparametric hypersurfaces are closely related to isoparametric functions. Geometrically, a smooth function is isoparametric if its regular level sets are parallel isoparametric hypersurfaces. For example, for any given point $x_0=(v_0,z_0,t_0)$ of the Damek--Ricci space, the function $F_{x_0}(x)=4t_0\cosh^2(d(x,x_0)/2)$ is an isoparametric function, the regular level sets of which are the concentric geodesic spheres centered at $x_0$. (Letter $d$ denotes the  distance function.) Though many other functions of the distance from $x_0$ would provide an isoparametric function with the same level sets, the choice of this function is justified by the property that using the standard coordinates $(v,z,t)$ on the half-space model, $F_{x_0}(v,z,t)$ has the form $G_{x_0}(v,z,t)/t$, where $G_{x_0}$ is a polynomial of degree $4$. More explicitly,
\begin{equation}\label{eq:distance_like}
    F_{(v_0,z_0,t_0)}(v,z,t)=\frac{1}{t}\left(\left(t+t_0+\left\|\frac{v-v_0}{2}\right\|^2\right)^2+ \left\|z-z_0+\frac12[v,v_0]\right\|^2\right).
\end{equation}
The main observation of \cite{Csikos_Horvath_isoparametric1} was that the polynomial $G_{x_0}$ is properly defined also when the point $x_0$ is moving out of the half-space model, and yields an isoparametric function $F_{x_0}=G_{x_0}/t$. When $x_0$ is on the boundary of the model, then the level sets of $F_{x_0}$ are parallel horospheres, when it is in the complementary open half-space, then the level sets are the sphere-like isoparametric hypersurfaces introduced in  \cite{Csikos_Horvath_isoparametric1}. 

It was also pointed out in \cite{Csikos_Horvath_isoparametric1}, that some of the isoparametric hypersurfaces constructed in  \cite{Ramos_Vazquez} can be obtained as limits of the level sets of $F_{x_0}$ as $x_0$ tends to a point at infinity. Furthermore, all examples of \cite{Ramos_Vazquez} can be obtained as the level sets of a function of the form $G(v,z,t)/t$, where $G$ is a quadratic polynomial. 

This background motivated the main objective of this paper. Our main goal is to classify all isoparametric functions on the Damek--Ricci space that can be written as $G(v,z,t)/t$ in terms of the standard coordinates on the half-space model of the space, for which $G$ is a polynomial function. The classification (Main Theorem \ref{thm:main}) includes all the examples of \cite{Ramos_Vazquez} and  \cite{Csikos_Horvath_isoparametric1}, and also provides some new examples.
 
By a fundamental theorem of Q.~M.~Wang \cite{Wang}, if an isoparametric function admits a singular level set, called a focal variety, then the regular level sets of the function are tubes about any of the focal varieties. In particular, any of the focal varieties determines all the level sets of the function. The focal varieties of the new isoparametric functions found in this paper, coincide with the focal varieties of some sphere-like isoparametric hypersurfaces lying in certain  ``smaller'' Damek--Ricci subspaces.  

The paper is structured as follows. In Section \ref{sec:preliminaries}, we collect definitions and facts on isoparametric functions, on Damek--Ricci spaces, and on their half-space models. A function is isoparametric if it is transnormal and satisfies the so-called Laplace condition.  In Section \ref{sec:isoparametric_equations}, we compute the equations for a polynomial $G(v,z,t)$ which
characterize the transnormality and the Laplace conditions for the quotient $F=G/t$. It is remarkable that the transnormality condition alone gives very strong restrictions on the polynomial $G$. In Section \ref{sec:transnormality}, we collect these restrictions to understand the structure of the polynomial $G$. In the next section, we investigate, that among those polynomials that give a transnormal function $G/t$, which ones satisfy the equations coming from the Laplace condition. This will lead us to the proof of our Main Theorem \ref{thm:main}.

Finally, Section \ref{sec:misc} gives an explicit description of the focal varieties of the new isoparametric functions and a formula for the mean curvatures of the regular level sets as a function of their distances from the focal variety. Though it is tempting to expect that isoparametric hypersurfaces obtained in this paper will lead to a full classification of such hypersurfaces in Damek--Ricci spaces, this goal has not been reached yet. Tubes about $\mathbb R\mathbf H^k$ in $\mathbb C\mathbf H^k$ is a known example of a family of parallel isoparametric hypersurfaces \cite{Montiel}, but they cannot be obtained as regular level sets of an isoparametric function of the form $G/t$, where $G$ is a polynomial. However they are level sets of an isoparametric function of the form $G/t^2$ for a suitable polynomial $G$. We compute the polynomial $G$ explicitly in Subsection \ref{subsec:6.3}. There is a natural way to construct analogous functions on Damek--Ricci spaces. It will be proved in Proposition \ref{prop:6.4} that among these analogous functions only the one defined on the complex hyperbolic space is isoparametric. This result suggests that the example of tubes about $\mathbb R\mathbf H^k$ in $\mathbb C\mathbf H^k$  might be a sporadic exception.

\section{Preliminaries\label{sec:preliminaries}}
\subsection{Isoparametric functions\label{sec:isoparametric}}

\begin{definition}
A smooth function $F\colon M\to \mathbb R$ defined on a Riemannian manifold $M$ is said to be \emph{isoparametric} if there exist a continuous function $a\colon \im F\to \mathbb R$ and a $\mathcal C^2$ function $b\colon \im F\to \mathbb R$ such that 
\begin{equation}\label{eq:isoparametric_a_b}
    \Delta F=a\circ F\qquad \text{ and }\qquad\|\nabla F\|^2=b\circ F.
\end{equation}
\end{definition}
We shall refer to the first condition $\Delta F=a\circ F$ as the \emph{Laplace condition}.
The geometric meaning of the second, so-called \emph{transnormality condition} is that nearby regular level sets of $F$ are parallel hypersurfaces. The mean curvature of a hypersurface $\Sigma$ can be expressed as $\frac{1}{\dim \Sigma}h$, where $h$ is the trace of the shape operator of $\Sigma$ (with respect to a fixed unit normal). The following proposition shows that the regular level sets of an isoparametric function are isoparametric hypersurfaces and gives a formula for their mean curvature.

\begin{proposition}[{\cite[Proposition 2.2]{Csikos_Horvath_isoparametric1}}]\label{prop:tube_mean_curvature}
 If $F$ is an isoparametric function  satisfying equations \eqref{eq:isoparametric_a_b}, then the trace $h$ of the shape operator of a regular level set $F^{-1}(c)$ of $F$ with respect to the unit normal vector field $\mathbf N=\frac{\nabla F}{\sqrt{b\circ F}}$ is expressed by
 \[
 h=\frac{-2 a(c)+b'(c)}{2\sqrt{b(c)}}.
 \]
 \end{proposition}

\begin{definition} The \emph{focal varieties} of an isoparametric function
    are its singular level sets. 
\end{definition}
There are some basic results of Q.~M.~Wang \cite{Wang} and  J.~Ge and Z.~Tang \cite{Ge_Tang}  on isoparametric functions.

\begin{theorem}[\cite{Wang},\cite{Ge_Tang}]\label{Wang}
For an isoparametric function $F$ on a connected and complete Riemannian manifold, 
\begin{enumerate}[label=\emph{(\roman*)}]
    \item only the minimal and maximal values of $F$ can be singular;
    \item the focal varieties of $F$ are smooth minimal submanifolds;
    \item the regular level sets of $F$ are tubes about either of the focal varieties, having  constant mean curvature.
\end{enumerate}
\end{theorem}	

\subsection{Damek--Ricci spaces\label{sec:Damek-Ricci}}
	
 Damek--Ricci spaces are solvable Lie groups equipped with a left-invariant Riemannian metric. To construct a Damek--Ricci space, we have to fix 
 \begin{itemize}
     \item a Euclidean linear space $(\mathfrak s,\langle\,,\rangle)$ with an orthogonal decomposition 
 $\mathfrak s=\mathfrak v\obot \mathfrak z\obot \mathfrak a$, where $\mathfrak a$ is a 1-dimensional subspace spanned by a given unit vector $\mathbf a\in \mathfrak a$;

 \item a representation $J\colon \mathrm{Cl}(\mathfrak z,q)\to \mathrm{End}(\mathfrak v)$ of the Clifford algebra $\mathrm{Cl}(\mathfrak z,q)$ of the quadratic form  $q\colon \mathfrak z\to \mathbb R$, $q(z)=-\langle z,z\rangle$ such that 
	\begin{equation*}
 	\|J_zv\|=\|z\| \|v\| \qquad \forall\, z\in \mathfrak z,\, v\in \mathfrak v.
	\end{equation*}
\end{itemize}
This equation implies also the identities 
\[\langle J_z v_1,J_z v_2\rangle=\|z\|^2\langle v_1,v_2\rangle\quad\text{and}\quad \langle J_zv_1,v_2\rangle =-\langle v_1,J_zv_2\rangle \qquad \forall\, z\in \mathfrak z,\, v_1,v_2\in \mathfrak v.
\]
  
We can equip the linear space $\mathfrak n=\mathfrak v\obot \mathfrak z$ with a Lie algebra structure such that $[\mathfrak n,\mathfrak z]=\{0\}$ and  $[\mathfrak v,\mathfrak v]\subseteq \mathfrak z$, defining the Lie bracket of $v_1,v_2\in \mathfrak v$ by
	\begin{equation} \label{eq:ad_J}
	\langle [v_1,v_2],z\rangle = \langle J_z v_1,v_2\rangle\qquad \forall\, z\in \mathfrak z.
	\end{equation}
If $\mathfrak v=\{0\}$ or $\mathfrak z=\{0\}$, then $\mathfrak n$ is commutative, otherwise  $\mathfrak n$ is a $2$-step nilpotent Lie algebra with center $\mathfrak z$. 

	We can introduce a solvable Lie algebra structure on $\mathfrak s$ with the Lie bracket
	\[
	[v_1+z_1+\tau_1\mathbf a,v_2+z_2+\tau_2\mathbf a]=\left(\frac{\tau_1}2v_2-\frac{\tau_2}2v_1 \right)+([v_2,v_1]+\tau_1z_2-\tau_2z_1) \quad \forall\, v_1,v_2\in \mathfrak v,\, z_1,z_2\in \mathfrak z,\,\tau_1,\tau_2\in\mathbb R.
	\]
    
	The simply connected, connected Lie group $S$ with Lie algebra $\mathfrak s$, equipped with the left invariant Riemannian metric induced by $\langle\,,\rangle$ is a Damek--Ricci space.
    
There is a classification of Damek--Ricci spaces. Every Damek--Ricci space is harmonic. The Damek--Ricci spaces corresponding to the degenerate case $\dim \mathfrak z=0$ are isometric with a real hyperbolic space $\mathbb R\mathbf H^{k}$. The further rank one symmetric spaces  $\mathbb C \mathbf H^k$, $\mathbb H \mathbf H^k$, $\mathbb O \mathbf H^2$ are also among the Damek--Ricci spaces, with $\dim \mathfrak z=1,3,7$, respectively, but none of the other Damek--Ricci spaces are symmetric. See \cite[Sections 3.1.2, 4.1.2, 4.4]{DamekRicci} for details.

\subsection{The half-space model of  Damek--Ricci spaces\label{sec:half_space_model}}
A convenient model of Damek--Ricci spaces can be built on the linear space $\mathfrak v\oplus \mathfrak z\oplus \mathbb R$ by pulling back the Riemannian metric of $S$ by the diffeomorphism $\Phi\colon\mathfrak v\oplus \mathfrak z\oplus \mathbb R \to S$ defined by 
\[
\Phi(v,z,\tau)=\exp(v+z)\exp(\tau \mathbf a),
\]
where $\exp$ is the exponential map of the Lie group $S$. From now on, we use the notation $n=\dim \mathfrak v$ and $m=\dim \mathfrak z$. Fixing orthonormal bases $\mathbf e_1,\dots,\mathbf e_n\in \mathfrak v$ and $\mathbf f_1,\dots,\mathbf f_m\in \mathfrak z$, the map $\Phi$ provides a global coordinate system $(v^1,\dots, v^n;z^1,\dots,z^m;\tau)\colon S\to \mathbb R^{n+m+1}$ by 
\[
\Phi^{-1}(p)=\left(\sum_{i=1}^nv^i(p)\mathbf e_i,\sum_{\alpha=1}^m z^{\alpha}(p)\mathbf f_{\alpha},\tau(p)\right) \qquad \text{ for }p\in S.
\]
The basis vector fields  induced by this chart on $S$ will be denoted by $\partial_{\mathbf e_1},\dots,\partial_{\mathbf e_n};\partial_{\mathbf f_1},\dots,\partial_{\mathbf f_m};\partial_{\tau}$.

We shall prefer to model the Damek--Ricci space on the open upper half-space $\mathfrak v\oplus \mathfrak z\times \mathbb R_+\subset \mathfrak v\oplus \mathfrak z\oplus \mathbb R$ obtained by the modification 
\[
\Psi\colon \mathfrak v\oplus \mathfrak z\times \mathbb R_+\to S,\qquad \Psi(v,z, t)=\Phi(v,z,\ln t)
\]
of the diffeomorphism $\Phi$. 
 

Furthermore, it is easy to rewrite known formulae computed in the model $S \cong \mathfrak v\oplus \mathfrak z\oplus \mathbb R$ to the half-space model by the simple coordinate transformation $t=e^{\tau}$, $\partial_t=e^{-\tau}\partial_{\tau}$.  

For example, rewriting the multiplication rule of the group $S$ computed in \cite[Section 4.1.3]{DamekRicci} to the half-space model $S\stackrel{\Psi}{\cong} \mathfrak v\oplus \mathfrak z\times \mathbb R_+$, we obtain  
\begin{equation*}\label{eq:multiplication}
    (v_1,z_1,t_1) \cdot (v_2, z_2, t_2)=\left(v_1+\sqrt{t_1}v_2,z_1+t_1z_2+\tfrac{1}{2}\sqrt{t_1}[v_1,v_2],t_1t_2\right).
\end{equation*}

It is clear from this equation that the left translation 
\begin{equation}\label{eq:left_translation}
    L_{(\bar v,\bar z,\bar t)}\big((v, z, t)\big)=\left(\sqrt{\bar t} v,\bar t z+\tfrac{1}{2}\sqrt{\bar t}\ad \bar v(v),\bar t t\right)+\left(\bar v,\bar z,0\right).
\end{equation}
by an arbitrary element $(\bar v,\bar z,\bar t)\in S$ extends to the whole space $\mathfrak v\oplus\mathfrak z\oplus\mathbb R$ as an affine transformation.

As the half-space model is embedded into the linear space $\mathfrak s$, there is a natural identification of the tangent space $T_{(\bar v,\bar z,\bar t)}S$ with $\mathfrak s$ for any 
$(\bar v,\bar z,\bar t)\in S$. Using this identification, the Riemannian metric of $S$ is given by the equation
\[
\langle (v_1,z_1,t_1),(v_2,z_2,t_2)\rangle_{(\bar v,\bar z,\bar t)}=\frac{1}{\bar t}\langle v_1,v_2\rangle+\frac{1}{\bar t^2}\left\langle z_1+\frac{1}{2}[v_1,\bar v],z_2+\frac{1}{2}[v_2,\bar v]\right\rangle+\frac{t_1t_2}{\bar t^2}.
\]

The left invariant vector fields  generated by the Lie algebra elements $\mathbf e_1,\dots,\mathbf e_n,\mathbf f_1,\dots,\mathbf f_m,\mathbf a$ can be expressed as follows (\cite[Section 4.1.5]{DamekRicci})
\begin{equation}\label{eq:left_inv}
  \mathbf{E}_i={\sqrt t}\left(\partial_{\mathbf e_i}-\tfrac12\partial_{[\mathbf e_i,v]} \right),
  \qquad
  \mathbf F_\alpha=t\partial_{\mathbf f_\alpha},
  \qquad
  \mathbf A =t\partial_t. 
  \end{equation}

The expression for the Laplace operator is
\begin{equation}\label{eq:Laplace}
	\Delta=t\Delta_{\mathfrak v}+t\left(t+\frac{\|v\|^2}{4}\right)\Delta_{\mathfrak z} +t^2\partial_{t}^2-\left(m+\frac{n}{2}-1\right)t\partial_{t}+t\sum_{i=1}^n \partial_{\mathbf e_i}\partial_{[v,\mathbf e_i]},
	\end{equation}
where $\Delta_{\mathfrak v}$ and $\Delta_{\mathfrak z}$ are the Euclidean Laplace operators on $\mathfrak v$ and $\mathfrak z$,  respectively. (See \cite[Section 4.4]{DamekRicci} and the first paragraph of the proof of Theorem 6.1 in \cite{Csikos_Horvath_isoparametric1}.)

\section{Isoparametric functions within a certain class of rational functions\label{sec:isoparametric_equations}}
From now on we work with the half-space model of a Damek--Ricci space. Let $\mathfrak F$ denote the family of function $F\colon S\to \mathbb R$ of the form $F(v,z,t)=\frac{G(v,z,t)}{t}$, where $G$ is a polynomial  function on $\mathfrak s$ not divisible by $t$.  The latter condition implies that $F$ is not constant. Our main goal is to find all isoparametric functions of this special type. We shall prove the following classification result.

\begin{maintheorem}\label{thm:main}
A function $F\in\mathfrak F$ on the half-space model of the Damek--Ricci space is isoparametric if and only if the polynomial $G=tF$ belongs to one of the following types:
\begin{enumerate}[label=\emph{(\roman*)},leftmargin=7mm]
    \item $G(v,z,t)=\tilde G(t)$ for a polynomial $\tilde G$ of $t$ such that $\tilde G(0) \neq 0$ and the function $\tilde G(t)/t$ has nowhere zero derivative on the positive half-line $(0,\infty)$.
    
    \item $G(v,z,t)=c_1\|\Pi_{\mathfrak w} v-w_0\|^2+c_2t$,  where $\Pi_{\mathfrak w}\colon \mathfrak v\to \mathfrak w$ is the orthogonal projection onto an arbitrary non-zero linear subspace $\mathfrak w\leq \mathfrak v$,  $w_0\in \mathfrak w$, $c_1\neq 0$ and $c_2$ are some constants. 
    
    \item 
  $G(v,z,t)=c_1\left(\left(t+\left\|\frac{v-v_0}{2}\right\|^2\right)^2+\|z-z_0+\tfrac{1}{2}[v,v_0]\|^2+ \lambda\langle v-v_0,(\Pi_{\mathfrak v_{+}}-\Pi_{\mathfrak v_{-}}) (v-v_0)\rangle+4\lambda^2\right)+c_2 t,$  
where $v_0\in \mathfrak v$, $z_0\in \mathfrak z$, $\lambda \geq 0$, $c_1\neq 0$, $c_2$ are constant vectors and numbers, respectively, $\Pi_{\mathfrak v_{\pm}}$ is the orthogonal projection onto $\mathfrak v_{\pm}$ for some decomposition $\mathfrak v=\mathfrak v_+\obot \mathfrak v_-$ of $\mathfrak v$ into the direct sum of the orthogonal $\mathrm{Cl}(\mathfrak z,q)$-submodules $\mathfrak v_{\pm}$.
\end{enumerate}
\end{maintheorem}
The family (i) contains exactly those isoparametric functions $F\in \mathfrak F$ that depend only on $t$. This fact is shown by the following lemma. 
\begin{lemma}
    For a polynomial $\tilde G\in \mathbb R[t]$, not divisible by $t$, the function $ F(v,z,t)={\tilde G}(t)/t$ is an isoparametric function on the half-space model of the Damek--Ricci space if and only if the function $\tilde G(t)/t$ has nowhere zero derivative on the positive half-line $(0,\infty)$.
\end{lemma}
\begin{proof}
The condition that $\tilde G(t)/t$ has nowhere zero derivative on the positive half-line $(0,\infty)$ is equivalent to the requirement that the function $\phi\colon (0,\infty)\to \mathbb R$, $\phi(t)={\tilde G}(t)/t$ is injective and its inverse $\phi^{-1}\colon \mathrm{im}(\phi)\to (0,\infty)$ is smooth. 

The function $t$ is isoparametric, and 
$
\|\nabla t\|^2=t^2$, $\Delta(t)= -\left(m+\frac{n}{2}-1\right)t$, thus, when $\phi'$ has no positive roots, we have $\Delta F=a\circ F$ and $\|\nabla F\|^2=b\circ  F$, where the functions $a,b\colon \textrm{im}\, \phi\to \mathbb R$ assign to $x=\phi(t)$ the numbers
\[a(x)= \phi''(t)t^2-\left(m+\frac{n}{2}-1\right)\phi'(t)t\quad\text{and}\quad b(x)=(\phi'(t)t)^2,
\]
therefore, $F$ is isoparametric.

Conversely, assume that the function $ F(v,z,t)={\tilde G}(t)/t$ is isoparametric for a polynomial $p\in \mathbb R[t]$, not divisible by $t$. For every $c\in \mathbb R$, the level set $F^{-1}(c)$ is a finite union of horospheres of the form
$t=t_i$, where the $t_i>0$ are the positive roots of
${\tilde G}(t)-ct$. By Theorem \ref{Wang}, every connected component of a
critical level set of $F$ must be a minimal submanifold. As horospheres have non-zero constant mean curvature, they are not minimal submanifolds, consequently, $F$ cannot have critical points. As $\|\nabla F\|(v,z,t)=|\phi'(t)|t$, we conclude that $\phi'(t)\neq 0$ for all $t>0$.
\end{proof}

Henceforth, we consider only isoparametric functions belonging to the family
\[
\mathfrak F^*=\{F\in \mathfrak F\mid F\text{ is not a function of $t$ alone}\}.
\]

\begin{definition}\label{def:equiv}
    Call two functions $F_1,F_2\colon S\to \mathbb R$ equivalent if there exist two real numbers $c_1\neq 0$ and $c_2$, and an element $(\bar v,\bar z,\bar t)\in S$ of the group $S$ such that
    \[
    F_2=c_1 F_1\circ L_{(\bar v,\bar z,\bar t)}+c_2.
    \]
\end{definition}
This is an equivalence relation on functions on $S$.   If $F(v,z,t)=G(v,z,t)/t\in \mathfrak F$, then the equivalence class of $F$ is contained in $\mathfrak F$. This follows at once from the composition rule 
\begin{equation}\label{eq:left_tranlation_of_F}
\begin{aligned}
F\circ L_{(\bar v,\bar z,\bar t)}(v,z,t)&=F\left(\sqrt{\bar t} v+\bar v,\bar t z+\bar z+\tfrac{1}{2}\sqrt{\bar t}[\bar v,v],\bar t t\right)\\&=\frac{1}{t}\left(\frac{1}{\bar t}G\left(\sqrt{\bar t} v+\bar v,\bar t z+\bar z+\tfrac{1}{2}\sqrt{\bar t}[\bar v,v],\bar t t\right)\right).
\end{aligned}
\end{equation}
As left translations of $S$ are isometries, either all functions belonging to an equivalence class are isoparametric or none of them are. Thus, to classify isoparametric functions in $\mathfrak F$, it is enough to find a representative of each equivalence class the elements of which are isoparametric.

In the remaining part of this section, we compute the equations characterizing isoparametric functions $F\in\mathfrak F$ in terms of the polynomial $G=tF$.

\subsection{The transnormality condition}
Differentiating $F$ with respect to the left invariant vector fields \eqref{eq:left_inv}, we get
\begin{align*}
  \mathbf{E}_i(F)={}&{\sqrt t}\left(\partial_{\mathbf e_i}F-\tfrac12\partial_{[\mathbf e_i,v]} F\right)=\frac{1}{\sqrt t}\left(\partial_{\mathbf e_i}G-\tfrac12\partial_{[\mathbf e_i,v]} G\right),
  \\
  \mathbf F_\alpha (F)={}&t\partial_{\mathbf f_\alpha} F=\partial_{\mathbf f_\alpha} G,
  \\
  \mathbf A (F)={} &t\partial_t (F)=\partial_t G-F. \end{align*}
  To compress our formulae, introduce the derivation $\mathbf D_i=\partial_{\mathbf e_i}-\frac12\partial_{[\mathbf e_i,v]}$.  
 Then we can express the squared norm of the gradient of $F$ as follows 
\begin{equation}\label{eq:gradient_F}
  \|\nabla F\|^2=\sum_{i=1}^n\frac{1}{t}\left( \mathbf D_iG\right)^2+\sum_{\alpha=1}^m(\partial_{\mathbf f_\alpha} G)^2 +\left(\partial_t G-F\right)^2,
\end{equation}
.

Thus, the orthonormality condition takes the form 
\[
\frac{1}{t}\left( -2 G\partial_tG+\sum_{i=1}^n\left(\mathbf D_iG\right)^2\right)+\sum_{\alpha=1}^m(\partial_{\mathbf f_\alpha} G)^2 +(\partial_t G)^2=b\circ F-F^2=\tilde b\circ F,
\]
where $b\colon \im F\to\mathbb R$ is a $\mathcal C^2$ function, $\tilde b$ is the function defined by $\tilde b(x)=b(x)-x^2$.

Consider the polynomial functions
\[
P_1=-2 G\partial_tG+\sum_{i=1}^n\left(\mathbf D_iG\right)^2\,\,\text{ and }\,\, P_2=\sum_{\alpha=1}^m(\partial_{\mathbf f_\alpha} G)^2 +(\partial_t G)^2.
\]
\begin{lemma}\label{lem:algebraic_dependence_orthonormality}
The rational functions $\frac{1}{t}P_1+P_2$ and $F$ are algebraically dependent.
\end{lemma}
\begin{proof} 
Since $\tilde b$ is a $\mathcal C^2$ function, the relation $\tfrac{1}{t}P_1+P_2=\tilde b\circ F$ implies that the differentials of $\tfrac{1}{t}P_1+P_2$ and $F$ are linearly dependent, thus, the statement follows from the Jacobi criterion (\cite[Section III.7, Theorem III]{Hodge_Pedoe_I}).
\end{proof}
\begin{lemma}\label{prop:algebraic_dependence}
Let $P(v,z,t)$ and $G(v,z,t)$ be arbitrary polynomial functions on $\mathfrak s$ such that $F=G/t\in \mathfrak F^*$ and $P$ and $F$ are algebraically dependent. Then $P$ is a constant polynomial.
\end{lemma}
\begin{proof} 
It will be convenient to consider $P$ and $G$ as elements of $\mathbb C[v^1,\dots,v^n,z^1,\dots,z^m,t]$. As $P$ and $F$ are algebraically dependent, they generate a subfield $K\leq \mathbb C(v^1,\dots,v^n,z^1,\dots,z^m,t)$ of transcendence degree 1 in the field of rational functions. By the extended L\"uroth theorem \cite{Igusa}, there is an element $h\in K$ such that $K=\mathbb C(h)$. Then we can find two rational functions $R_F,R_P$ in a single variable, such that $F=R_F(h)$ and $P=R_P(h)$.

Let us write $R_F$ as the quotient of two coprime polynomials and factorize both the numerator and the denominator into linear factors:
\[
R_F(z)=c\,\frac{\prod_{j=1}^k(z-\lambda_j)}{\prod_{j=1}^l(z-\mu_j)}
\qquad c\in \mathbb C,\, \lambda_j\neq \mu_{j'}.\]

Assume $h=A/B$ is the quotient of the coprime polynomials $A$ and $B$. Then equation $F=R_F(h)$ gives
\begin{equation}\label{eq:F=R(h)}
\frac{G}{t}=c\,\frac{B^{\max(0,l-k)}\prod_{j=1}^k(A-\lambda_jB)}{B^{\max(0,k-l)}\prod_{j=1}^l(A-\mu_jB)}.
\end{equation}
Since $t$ does not divide $G$ and $A$ is coprime to $B$, the numerator and the denominator of the fractions are coprime  on both sides of this equation. Hence, the denominator on the right hand side must have exactly one non-constant factor and the non-constant factor must be a constant multiple of $t$. 

Consider the case, when $B$ appears in the denominator, i.e., $k>l$, and $B$ is a constant multiple of $t$. Then the multiplicity of $B$ in the denominator must be $k-l=1$. We must have $l=0$, since otherwise the equations $B=c_1 t$ and $A-\mu_1 B=c_2$ for some constants $c_1$ and $c_2$ would imply that $A$ and $B$, as well as $G/t$ depend only on $t$, contradicting our assumptions. Thus, we have $k=1$ and $l=0$ in this case.

Now assume that one of the factors $A-\mu_jB$, say $A-\mu_1B$ is a constant multiple $c_1t$ of $t$. Then the multiplicity of the factor $A-\mu_1B$  in the denominator must be $1$. Furthermore,  neither $B$ nor any of the factors $A-\mu_jB$ $(j\geq 2)$ can be constant, as it would imply that $A$, $B$ and $G/t$ depend only on $t$. Consequently, in this case, we have $l=1\geq k$.

In both cases, equation \eqref{eq:F=R(h)} reduces to the form
\[
\frac{G}{t}=\frac{\alpha A+\beta B}{\gamma A+\delta B}.
\]
As $G/t$ is not constant, $\alpha\delta-\beta\gamma\neq 0$. Expressing $h=A/B$ from this equation, we get 
\[
P=R_P(A/B)=R_P\left(\frac{-\delta \frac{G}{t}+\beta}{\gamma \frac{G}{t}- \alpha}\right),
\]
therefore, $P=\tilde R_P(G/t)$ for some  rational function $\tilde R_P$. Proceeding as in the derivation of equation \eqref{eq:F=R(h)}, equation $P=\tilde R_P(G/t)$ yields an equation of the form
\begin{equation}\label{eq:P=T(G/t)}
P=\hat c\,\frac{t^{\max(0,\hat l-\hat k)}\prod_{j=1}^{\hat k}(G-\hat \lambda_j t)}{t^{\max(0,\hat k-\hat l)}\prod_{j=1}^{\hat l}(G-\hat \mu_jt)}\qquad \hat c\in \mathbb C, \hat \lambda_j\neq \hat \mu_{j'}.
\end{equation}

As $P$ is a polynomial, and the numerator and denominator of the fraction on the right hand side are coprime, all factors of the denominator are constant, this is possible only if $\hat l=\hat k=0$, however, in that case, the numerator is also constant, so $P$ is constant, as claimed.
\end{proof}

\begin{lemma}\label{prop:form_of_b} If $F\in \mathfrak F^*$, then we have $\tilde b(x)=b_1x+b_0$
 with some constants $b_1,b_0$. 
\end{lemma}
\begin{proof}
Let $R(x,y)$ be a non-zero polynomial in two variables such that  $R(\tfrac{1}{t}P_1+P_2,F)=0$. Decompose $R$ into homogeneous components $R=R_0+R_1+\dots+R_{\ell}$, where $R_i$ is a homogeneous polynomial of degree $i$,  $R_{\ell}\neq 0$. Multiplying the equation $R(\tfrac{1}{t}P_1+P_2,F)=0$ by $t^{\ell}$ and substituting $t=0$, we obtain the equation 
\[R_{\ell}(P_1(v,z,0),G(v,z,0))\equiv 0.\]
Assume $R_{\ell}(x,y)=\sum_{k=0}^{\ell}\alpha_k x^ky^{\ell-k}$. Then the above equation means that the values of the rational function $P_1(v,z,0)/G(v,z,0)$, (where they are defined), are roots of the polynomial $\sum_{k=0}^{\ell}\alpha_k x^k$. As this polynomial has only a finite number of roots, $P_1(v,z,0)/G(v,z,0)$ must be a constant function with some value $b_1$. This implies that the polynomial $P_1-b_1G$ is divisible by $t$ and can be written as $P_1-b_1G=t P_3$ with some polynomial $ P_3$. Since $\tfrac{1}{t}P_1+P_2$ and $F$ are algebraically dependent, so are $\tfrac{1}{t}P_1+P_2-b_1F= P_2+P_3$ and $F$. Then Lemma \ref{prop:algebraic_dependence} implies that $ P_2+P_3$ is a constant polynomial with some value $b_0$. Thus, $\tilde b\circ F=\tfrac{1}{t}P_1+P_2=b_1F+(P_2+P_3)=b_1F+b_0.$
\end{proof}

Write $G(v,z,t)$ as 
\[G(v,z,t)=\sum_{k=0}^rG_k(v,z)t^k,\] 
where $G_r(v,z)\neq 0$. For integers $k\notin [0,r]$, we set $G_k=0$. 

Then equating the coefficients of $t^k$ for $-1\leq k\leq  2r$ on the two sides of the equation
\[
\frac{1}{t}\left( -2 G\partial_tG+\sum_{i=1}^n\left(\mathbf D_iG\right)^2\right)+\sum_{\alpha=1}^m(\partial_{\mathbf f_\alpha} G)^2 +(\partial_t G)^2=b_1 F+b_0,
\]
we obtain the equation
\begin{equation}\label{eq:gradiens_k}
    \begin{aligned}
b_1G_{k+1}+&\delta_{0,k}b_0=\sum_{p=0}^{k+1}\left(-2(p+1)G_{k+1-p}G_{p+1} +  \sum_{i=1}^n \mathbf D_iG_p\,\mathbf D_iG_{k+1-p}\right)+\\&\qquad\qquad\qquad\sum_{p=0}^k\left(\sum_{\alpha=1}^m\partial_{\mathbf f_\alpha} G_p\,\partial_{\mathbf f_\alpha} G_{k-p} +(p+1)(k-p+1)G_{p+1}G_{k-p+1}\right)\\
=& \sum_{p=0}^{k+1}\left((p+1)(k-p-1)G_{k+1-p}G_{p+1} +  \sum_{i=1}^n \mathbf D_iG_p\,\mathbf D_iG_{k+1-p}\right)+\sum_{p=0}^{k}\sum_{\alpha=1}^m\partial_{\mathbf f_\alpha} G_p\,\partial_{\mathbf f_\alpha} G_{k-p}
,
\end{aligned}
\end{equation}
where $\delta_{0,k}$ is the Kronecker delta symbol.

\subsection{The Laplace condition}
Now we investigate the consequences of the Laplace condition for $F\in \mathfrak F^*$.
\begin{theorem}\label{thm:Laplace}
    If $F\in \mathfrak F^*$ satisfies the Laplace condition $\Delta F=a\circ F$ with some continuous function $a\colon \im F\to \mathbb R$, then $a$ is a linear function of the form $a(x)=(m+n/2+1)x+ a_0$, where $a_0$ is a constant.
\end{theorem}
Using equation \eqref{eq:Laplace}, the Laplace condition can be written as
\begin{equation}\label{eq:Laplace_F}
  \begin{aligned}
\Delta F&=\Delta_{\mathfrak v} G+\left(t+\frac{\|v\|^2}{4}\right)\Delta_{\mathfrak z} G+\left(m+\frac{n}{2}+1\right)\left(\frac{G}{t}-\partial_t G\right)+t\partial_t^2G+\sum_{i=1}^n \partial_{\mathbf e_i}\partial_{[v,\mathbf e_i]}G
=a\circ F.
\end{aligned}  
\end{equation}

Introduce the polynomial function $P\colon \mathfrak s\to \mathbb R$,
\[
P=\Delta_{\mathfrak v} G+\left(t+\frac{\|v\|^2}{4}\right)\Delta_{\mathfrak z} G-\left(m+\frac{n}{2}+1\right)\left(\partial_t G\right)+t\partial_t^2G+\sum_{i=1}^n \partial_{\mathbf e_i}\partial_{[v,\mathbf e_i]}G
\]
and the function $\tilde a\colon \im F\to \mathbb R$, $\tilde a(x)=a(x)-(m+n/2+1)x$. Then equation \eqref{eq:Laplace_F} gives 
$P=\tilde a\circ F$. 
\begin{lemma}\label{lem:algebraic_dependence}
The rational functions $P$ and $F$ are algebraically dependent.
\end{lemma}
\begin{proof} Since $F$ is not constant, there is a point $(v_0,z_0,t_0)\in S$ at which its gradient is non-zero, hence $F$ is a submersion of a neighborhood of the point $(v_0,z_0,t_0)$ onto an interval $I$ around $F(v_0,z_0,t_0)$. By the submersion theorem, there is a smooth curve $\gamma\colon I'\to S$ defined on an open subinterval $I'\subset I$, such that $F\circ \gamma=\iid_{I'}$. Then $\tilde a|_{I'}=P\circ \gamma$ is a \emph{smooth function}, thus the gradient vector fields $\nabla P=(\tilde a'\circ F)\nabla F$ and $\nabla F$ are linearly dependent at each point of the non-empty open set $F^{-1}(I')$. As the coordinate functions of these gradient vector fields are rational functions, this implies that the gradient vector fields are linearly dependent at each point of $\mathfrak s$, where they are both defined. Hence the Jacobi criterion of algebraic independence (\cite[Section III.7, Theorem II]{Hodge_Pedoe_I}) implies the statement. 
\end{proof}

\begin{proof}[Proof of Theorem \ref{thm:Laplace}]
According to Lemmata \ref{lem:algebraic_dependence} and  \ref{prop:algebraic_dependence}, $P\equiv a_0$ is a constant polynomial. Then equation \eqref{eq:Laplace_F} yields
\begin{equation}\label{eq:shifted_eigen}
\Delta F=P+\left(m+\frac{n}{2}+1\right)F=\left(m+\frac{n}{2}+1\right)F+a_0.\qedhere
\end{equation}
\end{proof}

Writing equation \eqref{eq:shifted_eigen} in terms of $G$, we obtain the equation
\begin{equation}\label{eq:Laplace_G}
\Delta_{\mathfrak v} G+\left(t+\frac{\|v\|^2}{4}\right)\Delta_{\mathfrak z} G-\left(m+\frac{n}{2}+1\right)\partial_t G+t\partial_t^2G+\sum_{i=1}^n \partial_{\mathbf e_i}\partial_{[v,\mathbf e_i]}G=a_0.
\end{equation}

Taking the coefficient of $t^k$ in this equation for $k=0,\dots,r,r+1$, we obtain the following system of equations on the polynomials $G_0,\dots,G_r$
\begin{align*}
\Delta_{\mathfrak v}G_{0}+\frac{\|v\|^2}{4}\Delta_{\mathfrak z}G_{0}-\left(m+\frac{n}{2}+1\right) G_{1}+\sum_{i=1}^n \partial_{\mathbf e_i}\partial_{[v,\mathbf e_i]}G_{0}&=a_0,
\\
\Delta_{\mathfrak v}G_{k}+\Delta_{\mathfrak z}G_{k-1}+\frac{\|v\|^2}{4}\Delta_{\mathfrak z}G_{k}-\left(m+\frac{n}{2}-k+1\right)(k+1) G_{k+1}+\sum_{i=1}^n \partial_{\mathbf e_i}\partial_{[ v,\mathbf e_i]}G_{k}&=0,
\\
\Delta_{\mathfrak v}G_r+\Delta_{\mathfrak z}G_{r-1}+\frac{\|v\|^2}{4}\Delta_{\mathfrak z}G_r+\sum_{i=1}^n \partial_{\mathbf e_i}\partial_{[ v,\mathbf e_i]}G_r&=0,
\\
\Delta_{\mathfrak z}G_r&=0.
\end{align*}
Let $\mathbf D$ be the differential operator
\[
\mathbf D=\Delta_{\mathfrak v}+\frac{\|v\|^2}{4}\Delta_{\mathfrak z}+\sum_{i=1}^n \partial_{\mathbf e_i}\partial_{[ v,\mathbf e_i]}.
\]
Then we have the recursion
\begin{align}
G_1&=\frac{1}{m+\frac{n}{2}+1} (\mathbf DG_0-a_0),\\
G_{k+1}&=\frac{1}{\left(m+\frac{n}{2}-k+1\right)(k+1)}\left(\mathbf D G_{k}+\Delta_{\mathfrak z}G_{k-1}\right),\label{eq:Laplacek}\\
0&=\mathbf DG_r+\Delta_{\mathfrak z} G_{r-1},\label{eq:Laplace1}\\
0&=\Delta_{\mathfrak z} G_{r}.\label{eq:Laplace2}
\end{align}
Given $G_0$, $G$ can be computed by the recursion, however, the choice of $G_0$ is restricted by the last two equations \eqref{eq:Laplace1}, \eqref{eq:Laplace2}.

\section{Consequences of the transnormality condition\label{sec:transnormality}}

In this section, we are interested in the consequences of the transnormality condition. For this reason, we assume only that $F(v,z,t)=G(v,z,t)/t\in \mathfrak F^*$ is transnormal, but it may not satisfy the Laplace condition.

\begin{proposition}\label{prop:G_r}\mbox{}
\begin{enumerate}[label=\emph{(\roman*)}]
    \item $G_r$ is a constant polynomial if $r\geq 1$.
    \item The degree $r$ of the variable $t$  in $G$ is at most $2$.
\end{enumerate}
   
\end{proposition}

\begin{proof} (i) If $r\geq 1$, then equation \eqref{eq:gradiens_k} for $k=2r$ yields
\[
0=\sum_{\alpha=1}^m(\partial_{\mathbf f_\alpha} G_r)^2,
\]
which means $\partial_{\mathbf f_\alpha} G_r=0$ for all $1\leq\alpha\leq m$. Having this in mind, equation \eqref{eq:gradiens_k} for $k=2r-1$ collapses to the equation
\[
0=\sum_{i=1}^n\left(\mathbf D_i G_r\right)^2,
\]
showing that 
\[
0=\mathbf D_i G_r=\partial_{\mathbf e_i}G_r-\tfrac12\partial_{[\mathbf e_i, v]} G_r=\partial_{\mathbf e_i}G_r
\]
for all $1\leq i\leq n$. This proves (i).

(ii)  Assume to the contrary that $r\geq 3$. Then equation \eqref{eq:gradiens_k} for $k=2r-2$ reduces to
\[
0=r(r-2)G_{r}^2+\sum_{\alpha=1}^m(\partial_{\mathbf f_\alpha} G_{r-1})^2,
\]
which cannot be fulfilled as the right-hand side is positive whenever $G_r\neq 0$. 
\end{proof}

\begin{theorem}
    Let $F\in \mathfrak F^*$ be a transnormal function for which the degree $r$ of $t$ in the polynomial $G=tF$  is at most one. Then  
    \[G(v,z,t)=c_1\|\Pi_{\mathfrak w} v-w_0\|^2+c_2t\] with some constants $c_1\neq 0$ and $c_2$, where $\Pi_{\mathfrak w}\colon \mathfrak s\to \mathfrak w$ is the orthogonal projection onto an arbitrary linear subspace $\mathfrak w\leq \mathfrak v$ and $w_0\in \mathfrak w$.
     
For any function $G$ of the above form, the function $F=G/t$ is isoparametric. 
\end{theorem} 
Note that the theorem recovers the isoparametric functions belonging to family (ii) of the Main Theorem \ref{thm:main}.
\begin{proof} Write $G$ as $G(v,z,t)=G_0(v,z)+G_1(v,z)t\neq 0$.
 $G_1$ is obviously constant if $G_1=0$. Otherwise, $G_1$ is constant by Proposition \ref{prop:G_r}. Set $G_1\equiv c_2$. $F$ is transnormal if and only if the function $F-c_2=\frac{G_0}{t}$ is transnormal, so to prove the restrictions on $G_0$, we may assume $c_2=0$ and $F(v,z,t)=\frac{G_0(v,z)}{t}$.

Under the assumption $r=0$, equation \eqref{eq:gradiens_k} for $k\in\{-1,0\}$, characterizing the transnormality of $F$ reduces to the following two equations
\begin{align*}
b_1 G_0&=  \sum_{i=1}^n\left(\partial_{\mathbf e_i}G_0-\tfrac12\partial_{[\mathbf e_i, v]} G_0\right)^2,
\\
b_0&=\sum_{\alpha=1}^m(\partial_{\mathbf f_\alpha} G_0)^2. 
\end{align*}
The second equation implies that $b_0$ must be non-negative, and all partial derivatives $\partial_{\mathbf f_\alpha} G_0$ must be constant polynomials, hence $G_0(v,z)=\langle z,z_0\rangle+f(v)$, where $z_0\in \mathfrak z$ is a vector of  length $\|z_0\|=\sqrt{b_0}$, and $f$ is a polynomial function on $\mathfrak v$. If the degree of $f$ is $s\geq 3$, then the right-hand side of the first equation has degree $2(s-1)$ in $v$, which is strictly greater than the degree of the left-hand side, so $f$ must have degree at most two.  Then $f$ can be written in the form $f(v)=\langle v,Lv\rangle+\langle v,v_0\rangle+f_0$, where $L\colon \mathfrak v\to \mathfrak v$ is a self-adjoint operator, $v_0\in \mathfrak v$, $f_0\in \mathbb R$. Then the first equation gives
\[
b_1\big(\langle v,Lv\rangle+\langle v,v_0\rangle+\langle z,z_0\rangle+f_0\big)=
\left\|2Lv+v_0+\tfrac{1}{2}J_{z_0}v\right\|^2.
\]
The right-hand side does not depend on $z$, thus the $z$ dependent term on the left-hand side should also vanish, giving $b_1z_0=0$.

\noindent \textbf {Case 1: $b_1=0$.} In this case, we have
\[
-\tfrac{1}{2}J_{z_0}v=2Lv+v_0.
\] 
Evaluating this equation at $v=0$, we obtain $v_0=0$ and consequently  $J_{z_0}=-4 L$. As $J_{z_0}$ is skew-adjoint, $L$ is self-adjoint, this equation implies $J_{z_0}=L=0$, that is, $z_0=0$. However, this shows that $F$ depends only on $t$, contradicting our assumption $F\in \mathfrak F^*$.

\noindent \textbf {Case 2: $b_1\neq 0$, $z_0=0$.}
In this case, we have to solve the equation 
\[
\langle v, L v\rangle +\langle v,v_0\rangle+f_0=
\frac{1}{b_1}\|2Lv+v_0\|^2=\left\langle v,\frac{4L^2}{b_1}v\right\rangle+\left\langle v,\frac{4L}{b_1}v_0\right\rangle+\frac{\|v_0\|^2}{b_1}.
\]
Equating the constant, linear and quadratic term of the two ends of this equation, we get
\[
f_0=\frac{\|v_0\|^2}{b_1},\qquad v_0= \frac{4L}{b_1}v_0,\qquad \frac{4L}{b_1}=\left(\frac{4L}{b_1}\right)^2.
\]
The third equation shows that the symmetric matrix $\frac{4}{b_1}L$ is an orthogonal projection $\Pi_{\mathfrak w}$ onto a linear subspace $\mathfrak w$ of $\mathfrak v$. Then the second equation  gives $\Pi_{\mathfrak w} v_0=v_0$, therefore $v_0\in \mathfrak w$. Thus, $G_0$ has the form
\[
G_0(v,z)=\frac{b_1}{4}\left\|\Pi_{\mathfrak w} v+\frac{2v_0}{b_1}\right\|^2,
\]
as claimed. 

The last part of the theorem claims that all functions of the form 
\begin{equation*}
    F(v,z,t)=\frac{1}{t}( c_1 \|\Pi_{\mathfrak w} v-w_0\|^2+c_2 t)=\frac{c_1}{t}\|\Pi_{\mathfrak w} v-w_0\|^2+c_2,
\end{equation*}
where $c_1, c_2$ are some constants, are isoparametric. 

To check the transnormality condition, we use equation $\eqref{eq:gradient_F}$ to obtain
\begin{equation}
  \|\nabla F\|^2=\frac{4c_1}{t}\|\Pi_{\mathfrak w} v-w_0\|^2+\left(c_2-F\right)^2=(F-c_2)^2+4(F-c_2).
\end{equation} 

The Laplace condition is equivalent to  equation \eqref{eq:Laplace_G}. Since $G$ does not depend on $z$, and $G$ is linear in $t$, the left-hand side of the equation can be simplified to 
\begin{equation}\label{eq:Laplace_G'}
\Delta_{\mathfrak v}( c_1 \|\Pi_{\mathfrak w} v-w_0\|^2+c_2 t)-\left(m+\frac{n}{2}+1\right)c_2=2c_1 \dim \mathfrak w-\left(m+\frac{n}{2}+1\right)c_2,
\end{equation}
so equation \eqref{eq:Laplace_G} holds if $a_0$ is chosen to be the constant on the right-hand side of \eqref{eq:Laplace_G'}.
\end{proof}

\begin{proposition}
    Assume that $F=G/t\in \mathfrak F^*$ is a transnormal function, where $G$ is a polynomial, in which the degree of the variable $t$ is equal to $2$. Then $G$ must be a non-zero constant multiple of a polynomial of the form
    \[
    t^2+\tfrac{1}{2}\left(\|\Pi_{\mathfrak w}(v-v_0)\|^2-b_1\right)t+ \langle z, \Pi_{\mathfrak y} z\rangle+\langle c(v), z\rangle+d(v),
    \]
    where $\Pi_{\mathfrak w}\colon \mathfrak v\to \mathfrak w$ is an orthogonal projection onto a linear subspace $\mathfrak w\leq \mathfrak v$, $v_0$ is a fixed element of $\mathfrak w$, $\Pi_{\mathfrak y}\colon \mathfrak z\to \mathfrak y$ is an orthogonal projection onto a linear subspace $\mathfrak y\leq \mathfrak z$, $c(v)$ and $d(v)$ are some polynomial functions on $\mathfrak v$.
\end{proposition} 
We note that not all functions $G$ of the given form will provide a transnormal function $F=G/t$.  

The proposition will follow from the rest of this section.

We may assume without loss of generality that the value of the constant polynomial $G_2\neq 0$ is $1$. The functions $G_0$ and $G_1$ satisfy equation \eqref{eq:gradiens_k} for $-1\leq k\leq 2$
\begin{align} \label{eq:k=-1}
b_1 G_0&=-2G_0G_1+  \sum_{i=1}^n\left(\partial_{\mathbf e_i}G_0-\tfrac12\partial_{[\mathbf e_i, v]} G_0\right)^2 
\\ \label{eq:k=0}
b_1G_1+b_0&=-G_1^2-4G_0+  \sum_{i=1}^n2\left(\partial_{\mathbf e_i}G_0-\tfrac12\partial_{[\mathbf e_i, v]} G_0\right)\left(\partial_{\mathbf e_i}G_1-\tfrac12\partial_{[\mathbf e_i, v]} G_1\right)+\sum_{\alpha=1}^m(\partial_{\mathbf f_\alpha} G_0)^2 
\\ \label{eq:k=1}
b_1&=-2G_{1}+ \sum_{i=1}^n\left(\partial_{\mathbf e_i}G_1-\tfrac12\partial_{[\mathbf e_i, v]} G_1\right)^2+
\sum_{\alpha=1}^m2\partial_{\mathbf f_\alpha} G_0\partial_{\mathbf f_\alpha} G_{1} 
\\ \label{eq:k=2}
0&= \sum_{\alpha=1}^m(\partial_{\mathbf f_\alpha} G_1)^2
\end{align}
Equation \eqref{eq:k=2} implies that $G_1(v,z)=G_1(v)$ does not depend on $z$, therefore equation $\eqref{eq:k=1}$ reduces to 
\begin{equation}\label{eq:G_1_t^1}
    \sum_{i=1}^n\left(\partial_{\mathbf e_i}G_1\right)^2=2G_1+b_1.
\end{equation}
If $G_1$ is constant, then \eqref{eq:G_1_t^1} gives that $G_1=-\frac{b_1}{2}$. If $G_1$ is not constant and its degree is $q>0$, then $\sum_{i=1}^n\left(\partial_{\mathbf e_i} G_1\right)^2$ has degree $2(q-1)$, and equation \eqref{eq:G_1_t^1} can be fulfilled only if $2(q-1)=q$, i.e., $q=2$. Thus, $G_1$ can be written in the form 
\[
G_1(v)=\tfrac{1}{2}\langle v,Av\rangle-\langle v_0,v\rangle+f_1,
\]
where $A\colon \mathfrak v\to\mathfrak v$ is a self-adjoint operator, $v_0\in \mathfrak v$, $f_1\in \mathbb R$. Substituting this form into equation \eqref{eq:G_1_t^1}, then we obtain
\[
\|Av-v_0\|^2=\langle v,Av\rangle-2\langle v_0,v\rangle+2f_1+b_1.
\]
Comparing the homogeneous parts of the two sides, we obtain the following system of equations
\begin{align*}
\|v_0\|^2=2f_1+b_1,\quad\qquad Av_0=v_0,\quad\qquad  A^2=A.
\end{align*}
Thus, $A$ is an orthogonal projection $\Pi_{\mathfrak w}$ of $\mathfrak v$ onto a linear subspace $\mathfrak w\leq \mathfrak v$, $v_0\in\mathfrak w$, and the general solution of equation \eqref{eq:G_1_t^1} has the form 
\[G_1(v)=\tfrac{1}{2}\left(\|\Pi_{\mathfrak w}(v-v_0)\|^2-b_1\right).\]
This family of general solutions of equation \eqref{eq:G_1_t^1} contains also the constant function $G_1=-\frac{b_1}2$ for $\mathfrak w=\{0\}$.

Now we turn our attention to the polynomial $G_0$, the choice of which is restricted by equations \eqref{eq:k=-1} and \eqref{eq:k=0}. Choose the basis $\mathbf e_1,...,\mathbf e_n$ so that $\mathfrak w$ is spanned by $\mathbf e_1,...,\mathbf e_{n'}$. Then equations \eqref{eq:k=-1} and \eqref{eq:k=0} take the form
\begin{gather}\label{eq:c0}
\|\Pi_{\mathfrak w}(v-v_0)\|^2 G_0=  \sum_{i=1}^n\left(\partial_{\mathbf e_i}G_0-\tfrac12\partial_{[\mathbf e_i, v]} G_0\right)^2, 
\\
\label{eq:c1}
  b_0-\tfrac{1}{4}b_1^2=-\tfrac{1}{4}\|\Pi_{\mathfrak w}(v-v_0)\|^4-4G_0+  \sum_{i=1}^{n'}2\left(\partial_{\mathbf e_i}G_0-\tfrac12\partial_{[\mathbf e_i, v]} G_0\right)\left\langle v-v_{0},\mathbf e_i\right\rangle+\sum_{\alpha=1}^m(\partial_{\mathbf f_\alpha} G_0)^2. 
\end{gather}
Assume that the highest degree non-zero $z$-homogeneous component $G_{0,q}$ of $G_0$ has degree $q\geq 3$. Then the degree $2(q-1)$ $z$-homogeneous component of the right-hand side of equation \eqref{eq:c1} is 
\[
\sum_{\alpha}(\partial_{\mathbf f_\alpha} G_{0,q})^2\neq 0,
\]
which contradicts the condition that the right-hand side must be equal to the constant polynomial $b_0-\frac{1}{4}b_1^2$. Thus, $G_0$ has degree at most $2$ in $z$. Let us write $G_0$ in the form
\[
G_0(v,z)=\langle z, B(v)z\rangle+\langle c(v), z\rangle+d(v),
\]
where $B(v)\colon\mathfrak z\to\mathfrak z$ is a self-adjoint linear map, $c(v)\in \mathfrak z$, and $d(v)$ as well as the entries of the matrix of $B(v)$ and the coordinates of $c(v)$ with respect to any basis of $\mathfrak z$ are polynomials of $v$.
\begin{lemma}\label{lem:B(v)}
$B(v)$ does not depend on $v$, and it is equal to an orthogonal projection $\Pi_{\mathfrak y}$ of $\mathfrak z$ onto a linear subspace $\mathfrak y\leq \mathfrak z$.     
\end{lemma}

\begin{proof}
The degree $2$ $z$-homogeneous component of equation \eqref{eq:c1} is
\[
 0=\langle z,(-4B(v)+2\partial_{\Pi_{\mathfrak w}(v-v_0)}B(v)) z\rangle+
 4\|B(v)z\|^2.
\]
This means that $B$ must satisfy the equation
\[
\partial_{\Pi_{\mathfrak w}(v-v_0)}B(v)=2B(v)-2B^2(v).
\]

Fix an arbitrary vector $v\in \mathfrak v$, and decompose $v-v_0$ as $v-v_0=w+w'$, where $w\in \mathfrak w$, $w'\in \mathfrak w^{\perp}$. Then the map $B_v\colon \mathbb R\to\End(\mathfrak z)$, 
\[
B_v(s)=B(v_0+w'+sw)
\]
satisfies the differential equation
\[
sB_v'(s)=\partial_{sw}B(v_0+w'+sw)= \partial_{\Pi_{\mathfrak w}((v_0+w'+sw)-v_0)}B(v_0+w'+sw)=2B_v(s)-2B_v^2(s).
\]
This differential equation has a unique solution defined on an open neighborhood of $1$, satisfying the initial condition $B_v(1)=B(v)$. It is easy to verify by substitution that the function 
\[
B_v(s)=s^2B(v)(I+(s^2-1)B(v))^{-1}
\]
solves the equation with exactly this initial condition, so this equation must hold in an open neighborhood of 1.

As $B(v)$ is self-adjoint, there is an orthonormal basis $\tilde {\mathbf f}_1,\dots,\tilde {\mathbf f}_m$ of $\mathfrak z$ consisting of eigenvectors of $B(v)$. Assume that $B(v)\tilde {\mathbf f}_{\alpha}=\lambda_{\alpha}\tilde {\mathbf f}_{\alpha}$ for $1\leq \alpha\leq m$. Then we have
\[
B_v(s)(\tilde {\mathbf f}_{\alpha})=\frac{s^2\lambda_{\alpha}}{1+(s^2-1)\lambda_{\alpha}}\tilde {\mathbf f}_{\alpha}\qquad \text{ for }1\leq \alpha\leq m.
\]
If $\lambda_{\alpha}$ is $0$ or $1$, then the $i$th eigenvalue of $B_v(s)$ is the constant $0$ or $1$ function, in any other cases, the $i$th eigenvalue of $B_v(s)$ is not a polynomial of $s$ as its meromorphic extension to $\mathbb C$ has poles at $s=\pm\sqrt{1-1/\lambda_{\alpha}}$.  By our assumption, the entries of the matrix $B_v(s)$ with respect to any basis of $\mathfrak z$ are polynomials of $s$, therefore the eigenvalues $\lambda_{\alpha}$ of $B(v)$ are all equal to $0$ or $1$, which means that $B(v)$ is the orthogonal  projection $\Pi_{\mathfrak y(v)}$ of  $\mathfrak z$ onto the eigenspace $\mathfrak y(v)$ of $B(v)$ corresponding to the eigenvalue $1$.

Then the squared Hilbert--Schmidt norm of $B(v)$ satisfies the inequality 
\[
\|B(v)\|_{\mathrm{HS}}^2=\tr  (B^2(v))=\sum_{\alpha=1}^m\lambda_\alpha^2\leq m,
\]
showing that the matrix elements of $B(v)$ with respect to an arbitrary basis are bounded functions of $v$. On the other hand, the matrix elements are polynomials of $v$, consequently, all matrix elements of $B(v)$ are constant, therefore $B(v)$ and $\mathfrak y(v)$ do not depend on $v$.
\end{proof}

\section{Incorporation of the Laplace condition}
In the previous section, we dealt with transnormal functions of the family $\mathfrak F^*$, and found strong restrictions on their possible algebraic forms. To complete the classification of isoparametric functions in $\mathfrak F^*$, we investigate what further restrictions are obtained on the form of $F\in\mathfrak F^*$ if, in addition to the transnormality condition, we impose also the Laplace condition. So from now on $F\in \mathfrak F^*$ is assumed to be an isoparametric function, and we use all the notations introduced in the previous section for the description of the algebraic form of $F$. 
\begin{lemma}
  The projections $\Pi_{\mathfrak w}$ and $\Pi_{\mathfrak y}$ are equal to the identical transformations $\Pi_{\mathfrak w}=\iid_\mathfrak v$ and $\Pi_{\mathfrak y}=\iid_\mathfrak z$, respectively, i.e., $\mathfrak w=\mathfrak v$ and $\mathfrak y=\mathfrak z$.
\end{lemma}
\begin{proof}
Equation \eqref{eq:Laplacek} for $k=1$ yields
\[
1=\frac{1}{2m+n}\left(\tr \Pi_{\mathfrak w}+2\tr \Pi_{\mathfrak y}\right).
\]
Since $\tr \Pi_{\mathfrak w}=\dim \mathfrak w\leq n$ and $\tr \Pi_{\mathfrak y}=\dim \mathfrak y\leq m$, equality can occur only if $\tr \Pi_{\mathfrak w}=n$ and $\tr \Pi_{\mathfrak y}= m$, i.e., $\Pi_{\mathfrak w}=\iid_{\mathfrak v}$ and $\Pi_{\mathfrak y}=\iid_{\mathfrak z}$.
\end{proof}

\begin{corollary}
If $G$ has degree $2$ in the variable $t$, and $G_2=1$, then $G$ has the form
\[
G(v,z,t)=t^2+\tfrac{1}{2}\left(\|v-v_0\|^2-b_1\right)t+\| z\|^2+\langle c(v), z\rangle+d(v).
\]
\end{corollary}
The next lemma focuses on the vector $c(v)$.
\begin{lemma} We have $c(v)=[v,v_0]+z_0$, where $z_0=c(v_0)\in \mathfrak z$.
\end{lemma}
\begin{proof}
The degree $1$ $z$-homogeneous component of equation \eqref{eq:c1} is
\begin{equation*}
  0=\left\langle -4c(v)+2\partial_{v-v_0}c(v)-  2[v-v_0,v]+4c(v),z \right\rangle, 
\end{equation*}
or equivalently 
\begin{equation}
  \partial_{v-v_0}c(v)=[v,v_0].
\end{equation}
 
 Fix a vector $v\in \mathfrak v$, and consider the map $c_{v}\colon \mathbb  R\to\mathfrak z$ defined by
\[
c_{v}(s)=c(sv+(1-s)v_0).
\]
The function $c_{v}$ satisfies the differential equation
\begin{align*}
sc_{v}'(s)&= \partial_{s(v-v_0)}c(sv+(1-s)v_0)=\partial_{(sv+(1-s)v_0)-v_0}c(sv+(1-s)v_0)=[sv+(1-s)v_0,v_0]=s[v,v_0].  
\end{align*}
The solution of this differential equation is the linear function $c_v(s)=s[v,v_0]+c_v(0)$. For $s=1$, this gives the statement.
\end{proof}

\begin{corollary}\label{lem:simplification_by_left_transl}
  If $G$ has degree $2$ in the variable $t$, and $G_2=1$, then $G$ has the form
\begin{equation}\label{eq:G_form}
G(v,z,t)=t^2+\tfrac{1}{2}\left(\|v-v_0\|^2-b_1\right)t+\| z\|^2+\langle [v,v_0]+z_0, z\rangle+d(v).
\end{equation} 
\end{corollary}
\begin{lemma}
     The equivalence class of $F$ (see Definition \ref{def:equiv}) contains a function $\tilde F$ such that bringing the polynomial $\tilde G=t\tilde F$ to the form analogous to \eqref{eq:G_form}, the parameter vectors $\tilde v_0\in \mathfrak v$ and $\tilde z_0\in \mathfrak z$ corresponding to $v_0$ and $z_0$ vanish.   
\end{lemma}
\begin{proof}
Equation \eqref{eq:left_tranlation_of_F} shows that the composition $\tilde F=F\circ L_{(v_0,-z_0/2,1)}$ is equal to $\frac{\tilde G(v,z,t)}{t}$, where $\tilde G$ is the polynomial
\begin{align*}
    \tilde G(v,z,t)&=G\left( v+v_0,z+\tfrac{1}{2}([v_0,v]-z_0), t\right)\\
    &=t^2+\tfrac{1}{2}\left(\|v\|^2-b_1\right)t+\| z+\tfrac{1}{2}([v_0,v]-z_0)\|^2+\langle [v,v_0]+z_0, z+\tfrac{1}{2}([v_0,v]-z_0)\rangle+d(v+v_0)\\
    &=t^2+\tfrac{1}{2}\left(\|v\|^2-b_1\right)t+\| z\|^2 +\left(\tfrac{1}{4}\|z_0\|^2-\left\|\tfrac{1}{2}[v,v_0]+z_0\right\|^2+d(v+v_0)\right)\\
    &=t^2+\tfrac{1}{2}\left(\|v\|^2-b_1\right)t+\| z\|^2 +\tilde d(v).
\end{align*}
Thus, $\tilde F$ has the required properties.
 \end{proof}
By the lemma, we may assume without loss of generality that the parameter vectors $v_0$ and $z_0$ appearing in equation \eqref{eq:G_form} vanish and 
\[
G(v,z,t)=t^2+\tfrac{1}{2}\left(\|v\|^2-b_1\right)t+\| z\|^2+d(v).
\]
It remains to determine which polynomials $d(v)$ of $v$ will produce an isoparametric function $F$. Plugging the function $G_0(v,z)=\|z\|^2+d(v)$ into equation \eqref{eq:c1} and using $\Pi_{\mathfrak w}=\iid_{\mathfrak v}$ and $v_0=0$, we obtain 
\[
  b_0-\tfrac{1}{4}b_1^2=-\tfrac{1}{4}\|v\|^4-4d(v)+  2\partial_{v}d(v). 
\]
Let $v$ be an arbitrary fixed vector of $\mathfrak v$ and consider the function $d_v\colon \mathbb R\to \mathbb R$, $d_v(s)=d(sv)$. Then $d_v$ satisfies the differential equation
\[
sd'_v(s)=\partial_{sv} d(sv)
=2d_v(s)+\tfrac{1}{2}\left(b_0-\tfrac{1}{4}b_1^2+\tfrac{1}{4}\|sv\|^4\right).
\]
Introduce the polynomial function $\ell\colon \mathfrak v\to\mathbb R$, $\ell(v)=d(v)-\frac{1}{16}\left(\|v\|^4 -4b_0+b_1^2\right)$. 
The above differential equation is solved by the function 
\[
s\mapsto\frac{1}{16}\|v\|^4 s^4+\ell(v) s^2 +\frac{b_1^2-4b_0}{16}\]
with initial value $d_v(1)$ at $s=1$, therefore it coincides with $d_v$. Consequently, we have
\begin{equation}\label{eq:c_v}
d(sv)=d_v(s)=\frac{1}{16}\|v\|^4 s^4+\ell(v) s^2 +\frac{b_1^2-4b_0}{16}.
\end{equation}
\begin{lemma}
    The polynomial $\ell$ is a homogeneous quadratic polynomial.
\end{lemma}
\begin{proof}
    Using the definition of $\ell$ and equation \eqref{eq:c_v}, we obtain
\begin{equation*}
 \ell(sv)=  d(sv)-\frac{1}{16}\left(\|v\|^4 s^4-4b_0+b_1^2\right)=s^2\ell(v),
\end{equation*}
as claimed.
\end{proof}
The lemma implies that there exists a self-adjoint operator $ Q\in\End(\mathfrak v)$ such that $\ell(v)=\langle v, Q v\rangle$. Thus, $d$ can be expressed as
\begin{equation*}
d(v)=\frac{1}{16}\|v\|^4 + \langle v, Q v\rangle+\frac{b_1^2-4b_0}{16}.
\end{equation*}
\begin{lemma}
 The operator $ Q$ satisfies the equation
 \[
  Q^2=\frac{b_1^2-4b_0}{64}\iid_{\mathfrak v}.
 \]
\end{lemma}
\begin{proof}
    Plug the function 
    \[
    G_0(v,z)=\|z\|^2+d(v)=\|z\|^2+\frac{1}{16}\|v\|^4 + \langle v, Q v\rangle+\frac{b_1^2-4b_0}{16}
    \] 
    into equation \eqref{eq:c0}. For the expression on the right-hand side, we obtain 
\begin{equation}\label{eq:gradiens_resz1}
  \partial_{\mathbf e_i}G_0-\tfrac12\partial_{[\mathbf e_i, v]} G_0 =\tfrac{1}{4}\|v\|^2\langle \mathbf e_i,v\rangle+2\langle \mathbf e_i, Q v\rangle-\langle [\mathbf e_i, v],z\rangle =\left\langle \mathbf e_i,\tfrac{1}{4}\|v\|^2v+2 Q v +J_zv\right\rangle,
\end{equation}
and
\begin{equation}\label{eq:26RHS}
    \sum_{i=1}^n\left(\partial_{\mathbf e_i}G_0-\tfrac12\partial_{[\mathbf e_i, v]} G_0\right)^2 =\left\|\tfrac{1}{4}\|v\|^2v+2 Q v+J_zv\right\|^2.
\end{equation}
Thus, using $\Pi_{\mathfrak w}=\iid_{\mathfrak v}$ and $v_0=0$, equation \eqref{eq:c0} gives
\begin{equation}\label{eq:c0'}
    \|v\|^2 (\|z\|^2+d(v))=  \left\|\tfrac{1}{4}\|v\|^2\ v +2 Q v+J_zv\right\|^2
\end{equation}
Substituting  $z=0$ yields
\[
 \|v\|^2 d(v)=\left\|\tfrac{1}{4}\|v\|^2\ v +2 Q v\right\|^2.
\]
Equating the degree 2 homogeneous parts of the two sides of this equation, we get
\[
\|v\|^2 \frac{b_1^2-4b_0}{16}=4\| Q v \|^2, 
\]
or equivalently,
\[
\left\langle v, \left(\frac{b_1^2-4b_0}{16}\iid_{\mathfrak v}-4 Q^2\right)v\right\rangle=0\qquad\text{ for all }v\in \mathfrak v. \qedhere
\]
\end{proof}
As the eigenvalues of the self-adjoint operator are real numbers, the lemma implies the inequality $b_1^2-4b_0\geq 0$.
\begin{corollary}\label{cor:5.8}
 The linear space $\mathfrak v$ has an orthogonal decomposition $\mathfrak v=\mathfrak v_+\obot \mathfrak v_-$ with orthogonal projections $\Pi_{\mathfrak v_{\pm}}\colon \mathfrak v\to \mathfrak v_{\pm}$ such that 
\[
 Q=\lambda\left(\Pi_{\mathfrak v_{+}}-\Pi_{\mathfrak v_{-}}\right),
\]  
where $\lambda=\frac{1}{8}\sqrt{b_1^2-4b_0}$ is the non-negative eigenvalue of $ Q$.
\end{corollary}
If $\lambda>0$, then $\mathfrak v=\mathfrak v_+\obot \mathfrak v_-$ is the eigenspace decomposition of $ Q$. If $\lambda=0$, then $ Q=0$ and the decomposition $\mathfrak v=\mathfrak v_+\obot \mathfrak v_-$ is not unique. In this case, we set $\mathfrak v_+=\mathfrak v$ and $\mathfrak v_-=\{\mathbf 0\}$.
\begin{proposition}
The subspaces $\mathfrak v_{\pm}$ are $\mathrm{Cl}(\mathfrak z,q)$-submodules of $\mathfrak v$.
\end{proposition}
\begin{proof}
The statement is trivial for $\lambda=0$, so consider the case $\lambda>0$. Equating the degree 1 $z$-homogeneous parts of \eqref{eq:c0'} gives 
\begin{equation*}
    0=  2\left\langle \tfrac{1}{4}\|v\|^2\ v +2 Q v,J_zv\right\rangle=4\left\langle   Q v,J_zv\right\rangle=4\left\langle[v, Q v], z\right\rangle
\end{equation*}
Let $v_+\in\mathfrak v_+$, $v_-\in\mathfrak v_-$, and $z\in\mathfrak z$ be arbitrary vectors, $v=v_++v_-$. Then adding the identities
\begin{align*}
\left\langle   v_+,J_zv_-\right\rangle&=\left\langle   v_+-v_-,J_zv_-\right\rangle\\
\left\langle   v_+,J_zv_-\right\rangle&=-\left\langle   v_-,J_zv_+\right\rangle=\left\langle   v_+-v_-,J_zv_+\right\rangle
\end{align*}
we get 
\[
2\left\langle   v_+,J_zv_-\right\rangle=\left\langle   v_+-v_-,J_zv\right\rangle=\frac{1}{\lambda}\left\langle   Q v,J_zv\right\rangle=\frac{1}{\lambda}\left\langle[v, Q v], z\right\rangle=0.
\]
Consequently, $\left\langle   v_+,J_zv_-\right\rangle=-\left\langle  J_z v_+,v_-\right\rangle=0$. Thus, we have $J_zv_+\in \mathfrak v_-^{\perp}=\mathfrak v_+$ and $J_zv_-\in \mathfrak v_+^{\perp}=\mathfrak v_-$, hence $\mathfrak v_+$ and $\mathfrak v_-$ are $\mathrm{Cl}(\mathfrak z,q)$-submodules of $\mathfrak v$.
\end{proof}

\begin{theorem}\label{thm:verification}
  For every polynomial function $G(v,z,t)$ of the form 
  \[
G(v,z,t)=t^2+\tfrac{1}{2}\left(\|v\|^2-b_1\right)t+\| z\|^2+\frac{1}{16}\|v\|^4 + \lambda\langle v,(\Pi_{\mathfrak v_{+}}-\Pi_{\mathfrak v_{-}}) v\rangle+4\lambda^2,
\]
where $\lambda\geq 0$, and $\Pi_{\mathfrak v_{\pm}}$ are orthogonal projections onto complementary orthogonal $\mathrm{Cl}(\mathfrak z,q)$-submodules $\mathfrak v_{\pm}$ of $\mathfrak v$, the function $F=G/t$ is isoparametric.
\end{theorem}
\begin{proof}
To prove the Laplace condition, we have to show that the left-hand side of \eqref{eq:Laplace_G}, i.e.,
\begin{equation*}
\Delta_{\mathfrak v} G+\left(t+\frac{\|v\|^2}{4}\right)\Delta_{\mathfrak z} G-\left(m+\frac{n}{2}+1\right)\partial_t G+t\partial_t^2G+\sum_{i=1}^n \partial_{\mathbf e_i}\partial_{[ v,\mathbf e_i]}G
\end{equation*}
is constant.
Computing the terms in this expression separately,
\begin{gather*}
 \Delta_{\mathfrak v} G=nt+\frac{n+2}{4}\|v\|^2 +2\lambda(n_+-n_-),  \qquad
 \Delta_{\mathfrak z} G=2m,\\
 \partial_t G=2t+\tfrac{1}{2}\left(\|v\|^2-b_1\right), \qquad
 \partial_t^2 G=2,\qquad
 \partial_{\mathbf e_i}\partial_{[ v,\mathbf e_i]}G=0,
\end{gather*}
where $n_{\pm}=\dim \mathfrak v_{\pm}$. Plugging these values into the main expression, we get 
\begin{align*}
nt&+\frac{n+2}{4}\|v\|^2 +2\lambda(n_+-n_-)+2m\left(t+\frac{\|v\|^2}{4}\right)-\left(m+\frac{n}{2}+1\right)\left(2t+\frac{1}{2}\left(\|v\|^2-b_1\right)\right)+2t\\
&=2\lambda(n_+-n_-)+\frac{1}{2}\left(m+\frac{n}{2}+1\right)b_1.
\end{align*}
As this is a constant, the Laplace condition holds.

    Now we check transnormality. Writing $G$ as a polynomial of $t$, the coefficient are 
\[
G_2\equiv 1,\qquad G_1(v,z)=\tfrac{1}{2}\left(\|v\|^2-b_1\right), \qquad G_0(v,z)=\| z\|^2+\frac{1}{16}\|v\|^4 + \langle v, Q v\rangle+4\lambda^2,
\]
where $ Q=\lambda(\Pi_{\mathfrak v_{+}}-\Pi_{\mathfrak v_{-}})$, as before. As $G_2$ is constant, the transnormality condition is equivalent to the system of equations \eqref{eq:k=-1}--\eqref{eq:k=2}. Consider first equation \eqref{eq:k=-1}.
Repeating the computation leading to equation \eqref{eq:26RHS}, the sum appearing in equation \eqref{eq:k=-1} equals 
\begin{equation*}
    \sum_{i=1}^n\left(\partial_{\mathbf e_i}G_0-\tfrac12\partial_{[\mathbf e_i, v]} G_0\right)^2 =\left\|\tfrac{1}{4}\|v\|^2v+2 Q v+J_zv\right\|^2.
\end{equation*}
Set $v_{\pm}=\Pi_{\mathfrak v_{\pm}}(v)$. As $\mathfrak v_+$ and $\mathfrak v_-$ are $J_z$-invariant orthogonal subspaces of $\mathfrak v$, we have the orthogonality relations 
\[J_zv\perp v \quad\text{ and }\quad J_zv=J_zv_++J_zv_-\perp \lambda(v_+-v_-)= Q v.
\]   
Thus,
\[
\left\|\tfrac{1}{4}\|v\|^2v+2 Q v+J_zv\right\|^2=
\left\|\tfrac{1}{4}\|v\|^2v+2 Q v\|^2+\|J_zv\right\|^2=\tfrac{1}{16}\|v\|^6+\langle v, Q v\rangle \|v\|^2+4\lambda^2\|v\|^2+\|z\|^2\|v\|^2,
\]
consequently,
\[
\sum_{i=1}^n\left(\partial_{\mathbf e_i}G_0-\tfrac12\partial_{[\mathbf e_i, v]} G_0\right)^2 =\|v\|^2G_0(v,z).
\]
Equation \eqref{eq:k=-1} follows from this by a simple algebraic manipulation.

Next we deal with equation \eqref{eq:k=0}. Observe that equation \eqref{eq:gradiens_resz1} holds in our situation as well. Combining this equation with the equation 
\begin{equation}\label{eq:gradient_resz2}
    \partial_{\mathbf e_i}G_1-\tfrac12\partial_{[\mathbf e_i, v]} G_1=\langle \mathbf e_i,v\rangle, 
\end{equation}
the first sum on the right-hand side of \eqref{eq:k=0} is
\[
\sum_{i=1}^n2\left(\partial_{\mathbf e_i}G_0-\tfrac12\partial_{[\mathbf e_i, v]} G_0\right)\left(\partial_{\mathbf e_i}G_1-\tfrac12\partial_{[\mathbf e_i, v]} G_1\right)=2\left\langle \tfrac{1}{4}\|v\|^2v+2 Q v +J_zv ,v \right\rangle=\tfrac{1}{2}\|v\|^4+4\langle v, Q v\rangle.
\]
We also have 
\[
\sum_{\alpha=1}^m(\partial_{\mathbf f_\alpha} G_0)^2=\sum_{\alpha=1}^m 4\langle \mathbf f_\alpha,z\rangle ^2=4\|z\|^2.
\]
Making use of the last two equations, a straightforward computation shows that equation \eqref{eq:k=0} holds with the choice $b_0=\frac{1}{4}b_1^2-16\lambda^2$ of the constant $b_0$.

Equation \eqref{eq:k=1} can be verified easily using 
equation \eqref{eq:gradient_resz2} and the fact that $G_1$ does not depend on $z$. Indeed,
\[
-2G_{1}+ \sum_{i=1}^n\left(\partial_{\mathbf e_i}G_1-\tfrac12\partial_{[\mathbf e_i, v]} G_1\right)^2+
\sum_{\alpha=1}^m2\partial_{\mathbf f_\alpha} G_0\partial_{\mathbf f_\alpha} G_{1} =-2G_{1}+ \|v\|^2=b_1.
\]

Finally, equation \eqref{eq:k=2} holds because $G_1$ does not depend on $z$.
\end{proof}

Following Lemma \ref{lem:simplification_by_left_transl}, we made some assumptions on the polynomial $G$ which could be achieved by a left translation with an element of the form $(v_0,z_0,1)$, and classified only those polynomials that satisfy these simplifying assumptions. We also used the assumption $G_2=1$, which could be attained by multiplying $G$ with a nonzero constant. Therefore, the general form of those polynomials $\tilde G$ which provide an isoparametric function $\tilde F(v,z,t)=\tilde G(v,z,t)/t$, and have degree $2$ in $t$ is 
\begin{align*}
  \tilde G(v,z,t)&=c_1 G\circ L_{(-v_0,-z_0,1)}(v,z,t)=c_1 G(v-v_0, z-z_0+\tfrac{1}{2}[v,v_0], t)\\
  &=c_1\left(\left(t+\left\|\frac{v-v_0}{2}\right\|^2\right)^2-b_1 t+\|z-z_0+\tfrac{1}{2}[v,v_0]\|^2+ \lambda\langle v-v_0,(\Pi_{\mathfrak v_{+}}-\Pi_{\mathfrak v_{-}}) (v-v_0)\rangle+4\lambda^2\right).  
\end{align*}
This completes the proof of the Main Theorem.

\section{Miscellaneous further results\label{sec:misc}}

\subsection{The focal variety of the new isoparametric functions\label{subsec:focal_variety}} The level sets of
the isoparametric functions of type (i) of the Main Theorem \ref{thm:main} are parallel horospheres, which have empty focal variety. The geometry of the level sets and the focal varieties of the isoparametric functions belonging to class (ii)  were studied in detail by D\'iaz-Ramos and Dom\'inguez-V\'azquez \cite{Ramos_Vazquez}, so here we deal only with isoparametric functions belonging to class (iii). We recall that the geometry of sphere-like isoparametric hypersurfaces, associated to certain functions of class (iii), were investigated  by the authors in \cite{Csikos_Horvath_isoparametric1}.  

Consider the level sets of a function 
\[
F(v,z,t)=\frac{c_1}{t}\left(\!\left(t+\left\|\frac{v-v_0}{2}\right\|^2\right)^{\!\!2}\!+\|z-z_0+\tfrac{1}{2}[v,v_0]\|^2+ \lambda\langle v-v_0,(\Pi_{\mathfrak v_{+}}\!\!-\Pi_{\mathfrak v_{-}}) (v-v_0)\rangle+4\lambda^2\right)+c_2,
\]
where the parameters are as in the Main Theorem \ref{thm:main} (iii). We may assume without loss of generality that $c_1=1$ and $c_2=0$, i.e.,
\begin{equation}\label{eq:F_def}
F(v,z,t)=\frac{1}{t}\left(\left(t+\left\|\frac{v-v_0}{2}\right\|^2\right)^{\!\!2}\!+\|z-z_0+\tfrac{1}{2}[v,v_0]\|^2+ \lambda\langle v-v_0,(\Pi_{\mathfrak v_{+}}\!\!-\Pi_{\mathfrak v_{-}}) (v-v_0)\rangle+4\lambda^2\right).
\end{equation}
To find the equation of the focal varieties of an isoparametric function, we have to find the extremal values of the function. Keeping $v$ and $z$ fixed $\lim_{t\to+\infty}F(v,z,t)=+\infty$, therefore $F$ has no maximum.
To search for its minimum, write $F$ in the form $F(v,z,t)=\frac{t^2+G_1(v,z)t+G_0(v,z)}t$.
\begin{lemma}
    The polynomial $G_0(v,z)$ is non-negative.
\end{lemma}
\begin{proof} Introducing the notation $N_{\pm}=\|\Pi_{\mathfrak v_{\pm}}(v-v_0)\|^2$, we have
\begin{align*}
    G_0(v,z)&=\frac{1}{16}\left\|v-v_0\right\|^4+\|z-z_0+\tfrac{1}{2}[v,v_0]\|^2+ \lambda\langle v-v_0,(\Pi_{\mathfrak v_{+}}-\Pi_{\mathfrak v_{-}}) (v-v_0)\rangle+4\lambda^2\\
    &\geq \frac{1}{16}\left\|v-v_0\right\|^4+ \lambda\langle v-v_0,(\Pi_{\mathfrak v_{+}}-\Pi_{\mathfrak v_{-}}) (v-v_0)\rangle+4\lambda^2\\
    &= \frac{1}{16}(N_{+}+N_{-})^2+ \lambda(N_{+}-N_{-})+4\lambda^2=N_{+}\left(\frac{1}{16}(N_{+}+2N_{-})+ \lambda \right)+\left(\frac{N_{-}}{4}-2\lambda\right)^2\geq 0 ,
\end{align*}
where in the last step, we used the assumption $\lambda\geq 0$.
\end{proof}
Observe that $G_1(v,z)=\frac{1}{2}(N_++N_-)$. By the lemma, keeping $v$ and $z$ fixed, $F(v,z,t)$ is minimized at $t=\sqrt{G_0(v,z)}$. Therefore,  
\[F(v,z,t)\geq F\left(v,z,\sqrt{G_0(v,z)}\right)=\frac{N_{+}+N_{-}}{2}+2\sqrt{G_0(v,z)}\geq \frac{N_{-}}{2}+2\left|\frac{N_{-}}{4}-2\lambda\right|\geq 4\lambda.\]
$F$ attains its least possible value $4\lambda$ if and only if 
\begin{align*}
z=z_0+\tfrac{1}{2}[v_0,v], \qquad N_{+}=0, \qquad N_{-}\leq 8\lambda,\quad \text{and} \quad t=2\lambda-\tfrac{1}{4}N_{-}.
\end{align*}
Writing these equations in terms of the variables $(v,z,t)$, we get the equation of the focal variety:
\[
z=z_0+\tfrac{1}{2}[v_0,v],\qquad \Pi_{\mathfrak v_{+}}(v-v_0)=0, \qquad t=2\lambda-\tfrac{1}{4}\|\Pi_{\mathfrak v_{-}}(v-v_0)\|^2.
\]

This is the intersection of a paraboloid of dimension $n_-$ and the model of the Damek--Ricci space, which is non-empty if $\lambda >0$, and becomes empty when $\lambda=0$. We remark that in the case $n_-=0$, the focal variety consists of a single point, and the regular level sets are concentric geodesic spheres. In the case $\lambda=0$, the level sets are parallel horospheres.

\begin{proposition}
 Assume $\lambda >0$. Let  $S_-$ be the Lie subgroup of $S$ corresponding to the Lie subalgebra $\mathfrak s_-=\mathfrak v_-\oplus \mathfrak z\oplus \mathfrak a$ of $\mathfrak s$. Then $S_-$ is also a Damek--Ricci space isometrically embedded into $S$. Set $v_0'=\Pi_{\mathfrak v_-}(v_0)$ and $v_0''=\Pi_{\mathfrak v_+}(v_0)$. Let $\mathcal F_{(v_0',z_0,-2\lambda)}\subset S_-$ be the focal variety of the distance-like isoparametric function $F_{(v_0',z_0,-2\lambda)}$ on $S_-$ defined by \eqref{eq:distance_like}. Then the focal variety $\mathcal F$ of the function $F$ is the left translate of $\mathcal F_{(v_0',z_0,-2\lambda)}$ by $(v_0'',0,1)$.
\end{proposition}

\begin{proof}
A point $(v,z,t)$ belongs to $L_{(v_0'',0,1)}\mathcal F_{(v_0',z_0,-2\lambda)}$ if and only if 
\[
L_{(v_0'',0,1)}^{-1}(v,z,t)=\left(v-v_0'',z-\tfrac{1}{2}[v_0'',v],t\right)\in \mathcal F_{(v_0',z_0,-2\lambda)}.
\]

It was computed in \cite[Section 7]{Csikos_Horvath_isoparametric1}, (where the function $F_{x_0}$ was denoted by $D_{x_0}$), that the equation of $\mathcal F_{(V_0,Z_0,t_0)}\subset S_-$ for $V_0\in \mathfrak v_-$ and $t_0<0$ is
\begin{equation*}
(V,Z,t)\in \mathfrak v_-\oplus \mathfrak z\times \mathbb R_+,\qquad    \left\|V-V_0\right\|^2=-4(t+t_0),\qquad Z=Z_0-\tfrac12[V,V_0].
\end{equation*}
In our case, $\left(v-v_0'',z-\tfrac{1}{2}[v_0'',v],t\right)$ belongs to $\mathcal F_{(v_0',z_0,-2\lambda)}$ if and only if
\begin{equation*}
v-v_0''\in \mathfrak v_-,\qquad    \left\|v-v_0''-v_0'\right\|^2=-4(t-2\lambda),\qquad z-\tfrac{1}{2}[v_0'',v]=z_0-\tfrac12[v-v_0'',v_0'].
\end{equation*}
The first condition $v-v_0''\in \mathfrak v_-$ is equivalent to $\Pi_{\mathfrak v_+}(v-v_0'')=\Pi_{\mathfrak v_+}(v-v_0'-v_0'')=\Pi_{\mathfrak v_+}(v-v_0)=0.$ Assuming $v-v_0\in \mathfrak v_-$, the second condition is equivalent to
\[
t=2\lambda-\tfrac{1}{4}\|v-v_0\|^2=2\lambda-\tfrac{1}{4}\|\Pi_{\mathfrak v_{-}}(v-v_0)\|^2.
\]
Finally, the third condition is equivalent to $z=z_0+\tfrac{1}{2}[v_0,v]$, since using the fact that $[\mathfrak v_+,\mathfrak v_-]=0$, we have $[v_0',v_0'']=0$ and
\[
[v_0'',v]-[v-v_0'',v_0']=[v_0'',v]+[v_0',v]=[v_0,v].
\]
This proves 
\[
(v,z,t)\in L_{(v_0'',0,1)}\mathcal F_{(v_0',z_0,-2\lambda)}\iff (v,z,t)\in \mathcal F.\qedhere
\]
\end{proof}

\subsection{The mean curvatures of the new isoparametric hypersurfaces}
Consider a regular level set $\Sigma_c=F^{-1}(c)$ of the function $F$ in \eqref{eq:F_def}. $\Sigma_c$ is a hypersurface of constant mean curvature, and it is also a tube of radius $r$ about the focal variety $\mathcal F$. Our goal is to express the mean curvature of $\Sigma_c$ as a function of the radius $r$. 
\begin{proposition}
    The trace $h$ of the Weingarten operator of $\Sigma_c$ with respect to the unit normal vector field $\frac{1}{\sqrt{b\circ F}}\nabla F$ equals
    \[h=-\left(m+\tfrac n2\right)\coth r-\tfrac12(n_+-n_-)\cosech r.\]
\end{proposition}
\begin{proof}
The Laplace condition for $F$ takes the form
\begin{equation*}
\Delta F=\left(m+\tfrac{n}{2}+1\right)F+a_0
\end{equation*}
by equation \eqref{eq:shifted_eigen}, where 
$$
a_0=2\lambda(n_+-n_-)+\left(m+\tfrac{n}{2}+1\right)b_1=2\lambda(n_+-n_-)
$$ 
according to equation \eqref{eq:Laplace_G} and 
the proof of Theorem \ref{thm:verification}. (Note that the constant $b_1$ appearing in Theorem \ref{thm:verification} vanishes in our case.)

The function $b$ appearing in the transnormality condition $\|\nabla F\|^2=b\circ F$ has the form $$b(x)=x^2+\tilde b(x)=x^2+b_1x+b_0=x^2+b_0$$ by Lemma \ref{prop:form_of_b}, where  $b_0=-16\lambda^2$ from the definition of $\lambda$ in Corollary \ref{cor:5.8}.

In Subsection \ref{subsec:focal_variety}, we saw that the minimal value of $F$ is $c_0=4\lambda$. Thus, we can express the radius $r$ of the tube $\Sigma_c$ as a function of $c$ as follows (see \cite{Wang}):
\[
r(c)=\int_{c_0}^c \frac{\id x}{\sqrt{b(x)}}=\int_{4\lambda}^c\frac{\mathrm dx}{\sqrt{x^2-16\lambda^2}}=\left[\mathrm{arcosh}\frac x{4\lambda}\right]_{4\lambda}^c=\mathrm{arcosh}\frac c{4\lambda}.\]
The inverse function is $c(r)=4\lambda \cosh r$.

We can compute the trace $h$ of the Weingarten operator of $\Sigma_c$ with respect to the unit normal vector field $\frac{1}{\sqrt{b\circ F}}\nabla F$ with the help of Proposition \ref{prop:tube_mean_curvature}. As
\[a(x)=(m+\tfrac n2+1)x+2\lambda(n_+-n_-),\qquad b(x)=x^2-16\lambda^2,\]
we get 
\begin{align*}
h&=\frac{-2 a(c)+b'(c)}{2\sqrt{b(c)}}=\frac{-2(m+\frac n2+1)c-4\lambda(n_+-n_-)+2c}{2\sqrt{c^2-16\lambda^2}}=\frac{-(m+\frac n2)c-2\lambda(n_+-n_-)}{\sqrt{c^2-16\lambda^2}}\\
&=\frac{-(m+\frac n2)4\lambda\cosh r-2\lambda(n_+-n_-)}{4\lambda\sinh r}
=-\left(m+\tfrac n2\right)\coth r-\tfrac12(n_+-n_-)\cosech r.\qedhere
\end{align*}    
\end{proof}

\subsection{An isoparametric function on $\mathbb C\mathbf H^k$ with focal variety $\mathbb R\mathbf H^k$\label{subsec:6.3}}
The complex hyperbolic space $\mathbb C\mathbf H^k$ is a Damek--Ricci space with $m=1$ and $n=2(k-1)$. It is convenient to represent the spaces $\mathfrak v$, $\mathfrak z$, and $\mathfrak a$ as $\mathfrak v=\mathbb C^{k-1}$ and $\mathfrak z=\mathfrak a=\mathbb R$. Then for $v\in \mathfrak v$ and $z\in \mathfrak z$, we have $J_{z}v= z\mathfrak i v$.

The real hyperbolic space $\mathbb R\mathbf H^k$ is embedded into $\mathbb C\mathbf H^k$ as a totally geodesic submanifold. It can be thought of as a degenerated Damek--Ricci space with components $\mathfrak v'=\mathbb R^{k-1}\leq \mathbb C^{k-1}=\mathfrak v$, $\mathfrak z'=0$, and $\mathfrak a'=\mathfrak a=\mathbb R$. It is known \cite[Theorem 7.4]{Montiel} that tubes about $\mathbb R\mathbf H^k$ in $\mathbb C\mathbf H^k$
are parallel isoparametric submanifolds. 

The map $\sigma\colon \mathbb C\mathbf H^k\to \mathbb C\mathbf H^k$, $\sigma(v,z,t)=(\bar v,-z,t)$ is a an involutive isometry of $\mathbb C\mathbf H^k$ with fixed point set $\mathbb R\mathbf H^k$, so geometrically, $\sigma$ is the reflection of $\mathbb C\mathbf H^k$ in $\mathbb R\mathbf H^k$. Given a point $(v,z,t)\in \mathbb C\mathbf H^k$, the closest point of $\mathbb R\mathbf H^k$ to $(v,z,t)$ is the midpoint of the geodesic segment connecting $(v,z,t)$ to $\sigma(v,z,t)$. Hence, denoting the distance of $(v,z,t)$ from $\mathbb R\mathbf H^k$ by $r(v,z,t)$, we have
\[
r(v,z,t)=\tfrac{1}{2}d((v,z,t),(\bar v,-z,t)).
\]
Applying \eqref{eq:distance_like}, we obtain
\[
\cosh^2(r(v,z,t))=\cosh^2\left(\tfrac{1}{2}d((v,z,t),(\bar v,-z,t))\right)=\frac{1}{t^2}\left(\left(t+\tfrac12\left\|\mathrm{Im}\, v\right\|^2\right)^2+\left\|z-\tfrac12[\mathrm{Re}\, v,\mathbf {i}\,\mathrm{Im}\, v]\right\|^2\right).
\]
The resulting isoparametric function 
\begin{equation}\label{eq:RHn--CHn}
    F(v,z,t)=\frac{1}{t^2}\left(\left(t+\tfrac12\left\|\mathrm{Im}\, v\right\|^2\right)^2+ \left\|z-\tfrac12[\mathrm{Re}\, v,\mathbf{i}\,\mathrm{Im}\, v]\right\|^2\right)
\end{equation}
has the form of a quartic polynomial divided by $t^2$, and for this reason, it is not contained in our classification. Furthermore, it is not the square of any of the functions in the classification. In the next subsection, we show that isoparametric functions of similar type cannot exist in Damek--Ricci spaces different from $\mathbb C\mathbf H^k$. 

\subsection{An attempt to generalize isoparametric function \eqref{eq:RHn--CHn}}
Consider an arbitrary Damek--Ricci space and a decomposition $\mathfrak v=\mathfrak v_1\operp \mathfrak v_2$ of $\mathfrak v$ into the direct sum of two orthogonal subspaces. Denote by $\Pi_i\colon\mathfrak v\to\mathfrak v_i$ the orthogonal projection onto $\mathfrak v_i$ ($i=1,2$) and consider the function 
\begin{equation}\label{eq:general_F}
    F(v,z,t)=\dfrac{\left(t+\frac12\|\Pi_2v\|^2\right)^2+\left\|z-\tfrac12[\Pi_1v,\Pi_2v]\right\|^2}{t^2}.
\end{equation}
Clearly, $F$ is a generalization of the function defined in \eqref{eq:RHn--CHn}. The following proposition tells us under what conditions this function $F$ will be transnormal or isoparametric.

\begin{proposition}\label{prop:6.4}\mbox{}
\begin{enumerate}[label=\emph{(\roman*)}]
    \item The function $F$ defined above is transnormal if and only if $J_{\mathfrak z}\mathfrak v_1\subseteq \mathfrak v_2$ and $J_{\mathfrak z}\mathfrak v_2\subseteq \mathfrak v_1$. 
    \item $F$ is isoparametric if and only if the Damek--Ricci space is the complex hyperbolic space $\mathbb C\mathbf H^k$ with $n=2(k-1)$ and $\mathfrak v_1$ is a maximal totally real subspace of $\mathfrak v=\mathbb C^{k-1}$.
\end{enumerate}   
\end{proposition}

\begin{proof}
Denote by $n_i$ the dimension of $\mathfrak v_i$ for $i=1,2$. Choose the orthonormal basis $\mathbf e_1,\dots,\mathbf e_n$ of $\mathfrak v$ in such a way that $\mathbf e_1,\dots,\mathbf e_{n_1}$ span $\mathfrak v_1$, and  $\mathbf e_{n_1+1},\dots,\mathbf e_{n}$ span $\mathfrak v_2$. Let $G(v,z,t)=t^2F(v,z,t)$ be the polynomial in the numerator of $F$. The derivatives of $F$ with respect to the left-invariant vector fields \eqref{eq:left_inv} can be computed straightforwardly. If $i\leq n_1$, then
\begin{align*}
    \mathbf E_i(F)&=t^{-3/2}\left\langle z-\tfrac12[\Pi_1v,\Pi_2v],-[\mathbf e_i,\Pi_2v+v]\right\rangle
=t^{-3/2}\left \langle J_{ z-\tfrac12[\Pi_1v,\Pi_2v]}(\Pi_2v+v),\mathbf e_i\right\rangle.
\end{align*}
For $i> n_1$, we obtain
\begin{align*} 
    \mathbf E_i(F)&=t^{-3/2}\left((2t+\|\Pi_2v\|^2)\langle\Pi_2v,\mathbf e_i\rangle+\left\langle z-\tfrac12[\Pi_1v,\Pi_2v],-[\mathbf e_i,v-\Pi_1v]\right\rangle\right)\\
&=t^{-3/2}\left \langle (2t+\|\Pi_2v\|^2)\Pi_2v+J_{ z-\tfrac12[\Pi_1v,\Pi_2v]}\Pi_2v,\mathbf e_i\right\rangle.
\end{align*}
Furthermore,
\begin{align*}
\mathbf F_\alpha(F)&=\tfrac1t\partial_{\mathbf f_\alpha}G=\tfrac2t\langle z-\tfrac12[\Pi_1v,\Pi_2v],\mathbf f_\alpha\rangle,\\
\mathbf A(F)&=\tfrac1t\partial_tG-2F=\tfrac2t(t+\tfrac12\|\Pi_2v\|^2)-2F.
\end{align*}
Then
\begin{align*}
\|\nabla F\|^2
&=t^{-3}\underbrace{\left(\left\|\Pi_1J_{ z-\tfrac12[\Pi_1v,\Pi_2v]}(\Pi_2v+v)\right\|^2+\left\|(2t+\|\Pi_2v\|^2)\Pi_2v+\Pi_2J_{ z-\tfrac12[\Pi_1v,\Pi_2v]}\Pi_2v\right\|^2\right)}_{\text{denote by }E}
\\
&\qquad+\frac4{t^2}\left\|z-\tfrac12[\Pi_1v,\Pi_2v]\right\|^2+\left(\frac2t\left(t+\tfrac12\|\Pi_2v\|^2\right)-2F\right)^2\\
&=t^{-3}E+4F-8F-\frac{4}{t}\|\Pi_2v\|^2F+4F^2=t^{-3}\left(E-4\|\Pi_2v\|^2 t^2F\right)+4F^2-4F
\end{align*}
Introduce the polynomial function 
\begin{equation*}
H(v,z,t)=E-4\|\Pi_2v\|^2 G(v,z,t).
\end{equation*}
To prove the only if part of (i), assume that $F$ is transnormal. Then $F=G/t^2$ and $\|\nabla F\|^2-4 F^2+4F=H/t^3$ are  algebraically dependent rational functions by the Jacobi criterion.
Let $R(x,y)=\sum_{i+j\leq \ell}r_{ij}x^iy^j$ be a polynomial of degree $\ell>0$ such that $R(G/t^2,H/t^3)=0$. The polynomials $G$ and $H$ are both quadratic in $z$ with degree $2$ $z$-homogeneous parts
\begin{align*}
G_2(v,z,t)&=\|z\|^2,\\
H_2(v,z,t)&=\|\Pi_1J_{z}(\Pi_2v+v)\|^2+\|\Pi_2J_{z}\Pi_2v\|^2-4\|\Pi_2v\|^2 \|z\|^2.
\end{align*}
Thus, the degree $2\ell$ $z$-homogeneous part of $R(G/t^2,H/t^3)$ equals
\[
t^{-3\ell}\sum_{i+j=\ell}r_{ij}G_2^iH_2^jt^i.
\]
Since $R(G_2/t^2,H_2/t^3)=0$, furthermore, $G_2$ and $H_2$ do not depend on $t$,
$G_2^iH_2^j=0$ for all $i,j$ such that $i+j=\ell$ and $r_{ij}\neq 0$. As $R$ has degree $\ell$, there exists a pair of indices such that $G_2^iH_2^j=0$. $G_2$ is a non-zero polynomial, thus $H_2=0$, that is
\[
\|\Pi_1J_{z}(\Pi_2v+v)\|^2+\|\Pi_2J_{z}\Pi_2v\|^2=4\|\Pi_2v\|^2 \|z\|^2.
\]
Evaluating this equation at a vector $v\in \mathfrak v_1$, we get $\Pi_1J_zv=0$, therefore, $J_{\mathfrak z}\mathfrak v_1\subseteq \mathfrak v_1^{\perp}=\mathfrak v_2$. Assuming $v\in \mathfrak v_2$, the equation gives 
\[
4\|\Pi_1J_zv\|^2+\|\Pi_2J_zv\|^2=4\|z\|^2\|v\|^2.
\]
On the other hand, we know that
\[
4\|\Pi_1J_zv\|^2+4\|\Pi_2J_zv\|^2=4\|J_zv\|^2=4\|z\|^2\|v\|^2.
\]
Comparison of the last two equations shows $\Pi_2J_zv=0$, consequently $J_{\mathfrak z}\mathfrak v_2\subseteq \mathfrak v_1$. 

Conversely, conditions $J_{\mathfrak z}\mathfrak v_1\subseteq \mathfrak v_2$ and $J_{\mathfrak z}\mathfrak v_2\subseteq \mathfrak v_1$ allow us to simplify the formula for $\|\nabla F\|^2$ as
\begin{align*}
\|\nabla F\|^2
&=t^{-3}\Big(4\|J_{ z-\tfrac12[\Pi_1v,\Pi_2v]}\Pi_2v\|^2+\|(2t+\|\Pi_2v\|^2)\Pi_2v\|^2-4\|\Pi_2v\|^2 t^2F\Big)+4F^2-4F\\
&=t^{-3}\|\Pi_2v\|^2\Big(4\| z-\tfrac12[\Pi_1v,\Pi_2v]\|^2+(2t+\|\Pi_2v\|^2)^2-4t^2F\Big)+4F^2-4F\\
&=4F^2-4F.
\end{align*}
This completes the proof of (i). We note that the conditions $J_{\mathfrak z}\mathfrak v_1\subseteq \mathfrak v_2$ and $J_{\mathfrak z}\mathfrak v_2\subseteq \mathfrak v_1$ are equivalent to the conditions that $\mathfrak v_1$ and $\mathfrak v_2$ are commutative Lie subalgebras of $\mathfrak s$.

To show (ii), we compute the Laplacian of $F$. The computation will use the identity 
\[
\sum_{i=1}^n\|[\mathbf e_i,v]\|^2=m \|v\|^2 \qquad \forall\, v \in\mathfrak v.
\]
(See \cite[Lemma 3.6]{Csikos_Horvath_isoparametric1}.) By the commutativity of $\mathfrak v_1$ and $\mathfrak v_2$, this also implies 
\[
\sum_{i=n_1+1}^n\|[\mathbf e_i,\Pi_1 v]\|^2=\sum_{i=1}^n\|[\mathbf e_i,\Pi_1 v]\|^2=m \|\Pi_1v\|^2 \text{ and }\sum_{i=1}^{n_1}\|[\mathbf e_i,\Pi_2 v]\|^2=\sum_{i=1}^{n}\|[\mathbf e_i,\Pi_2 v]\|^2=m \|\Pi_2v\|^2
\]
for all $v \in\mathfrak v$.

According to \eqref{eq:Laplace}, $\Delta F$ is built from the following components
\begin{align*}
t\Delta_{\mathfrak v} F(v,z,t)&=\frac1t\left(tn_2+(n_2+2)\|\Pi_2v\|^2\right)+\frac{m\|v\|^2}{2t}\\
t\left(t+\frac{\|v\|^2}{4}\right)\Delta_{\mathfrak z} F(v,z,t)&= 2+\frac{\|v\|^2}{2t}\\
t^2\partial_{t}^2 F(v,z,t)&=\partial^2_tG(v,z,t)-4\frac{\partial_t G(v,z,t)}{t}+6\frac{G(v,z,t)}{t^2}=2-\frac{8t+4\|\Pi_2v\|^2}t+6F(v,z,t)\\
t\partial_{t} F(v,z,t)&=\frac{\partial_tG(v,z,t)}{t}-2\frac {G(v,z,t)}{t^2}=\frac{2t+\|\Pi_2v\|^2}t-2F(v,z,t)\\
t\sum_{i=1}^n \partial_{\mathbf e_i}\partial_{[v,\mathbf e_i]} F(v,z,t)
&=\frac1t\sum_{i=1}^{n_1}\langle-[\mathbf e_i,\Pi_2v],[v,\mathbf e_i]\rangle+\frac1t\sum_{i=n_1+1}^{n_2}\langle-[\Pi_1v,\mathbf e_i],[v,\mathbf e_i]\rangle=m\frac{\|\Pi_2v\|^2-\| \Pi_1v\|^2}{t}
\end{align*}
Putting the Laplacian of $F$ together from these pieces, after some algebraic simplifications we get
\begin{equation*}
\Delta F(v,z,t)
=(2m+n+4)F(v,z,t)-(2m+n_1+2)+\frac{(n_2-1-\frac{n}{2})\|\Pi_2v\|^2}t+\frac{(m+1)\|v\|^2}{2t}-\frac{m\|\Pi_1v\|^2}t.
\end{equation*}
If $v\in\mathfrak v_1$, then $F(v,z,t)=1+\|z\|^2/t^2$ depends only on $z$ and $t$. Consequently, if $\Delta F=a\circ F$ for some function $a$, then $\Delta F(v,z,t)$ must not depend on $v$ either. The formula for $\Delta F(v,z,t)$ for $v\in \mathfrak v_1$ contains the $v$ dependent part 
\[
\frac{(m+1)\|v\|^2}{2t}-\frac{m\|\Pi_1v\|^2}t=\frac{(1-m)\|v\|^2}{2t},
\]
hence $F$ can satisfy the Laplace condition only if $m=1$, which means that the ambient Damek--Ricci space is $\mathbb C\mathbf H^k$. 

In the case, when the ambient space is $\mathbb C\mathbf H^k$, and $\mathfrak v_1$ is a maximal totally real subspace of $\mathfrak v$, the function $F$ defined in \eqref{eq:general_F} is the composition of the isoparametric function defined in \eqref{eq:RHn--CHn} with an isometry of $\mathbb C\mathbf H^k$, therefore $F$ is isoparametric. 
\end{proof}
\begin{rem}
    The Clifford module $\mathfrak v$ can have an orthogonal decomposition $\mathfrak v=\mathfrak v_1\operp \mathfrak v_2$ in many cases different from the case $m=1$. See \cite[Chapter I, Proposition 3.6]{Lawson_Michelsohn} for a multitude of further examples.
\end{rem}


\bibliographystyle{acm}
\bibliography{isoparametric}
\end{document}